\documentclass[11pt]{amsart}

\usepackage{hyperref}
\usepackage{amssymb,amsmath}
\usepackage[alphabetic,nobysame,abbrev]{amsrefs}
\usepackage{enumerate}
\usepackage{color}
\RequirePackage{fix-cm}
\usepackage{bm}
\usepackage{mdframed}
\usepackage{cancel}
\usepackage{makecell}
\usepackage{diagbox}
\usepackage{pgf,tikz}
\usepackage{mathrsfs}
\usetikzlibrary{arrows}
\usetikzlibrary{arrows,decorations.markings}
\usetikzlibrary{matrix}
\usetikzlibrary{graphs}
\usetikzlibrary{backgrounds}

\definecolor{qqqqff}{rgb}{0.,0.,1.}
\definecolor{xdxdff}{rgb}{0.49019607843137253,0.49019607843137253,1.}

\definecolor{qqqqff}{rgb}{0.,0.,1.}

\usepackage{tikz}
\usepackage{graphicx}
\usepackage[cmtip,all]{xy}

\usepackage[draft]{todonotes}
\usepackage[cmtip,all]{xy}

\usepackage{hyperref}
\usepackage[utf8]{inputenc}
\usepackage{amsmath,amsfonts,amssymb,euscript,graphicx,color,tikz}

\theoremstyle{plain}
\newtheorem{thm}[equation]{Theorem}
\newtheorem{pro}[equation]{Proposition}
\newtheorem{cor}[equation]{Corollary}
\newtheorem{lem}[equation]{Lemma}
\theoremstyle{remark}
\newtheorem{rem}[equation]{\bf Remark}
\newtheorem{exa}[equation]{\bf Example}

\newtheorem{DEF}[equation]{\bf Definition}

\newtheorem{tab}[equation]{\hphantom{\quad\quad\quad\quad\quad\quad}Table}

\numberwithin{equation}{section}

\newcommand{\sub}{\subseteq}

\newcommand{\la}{Lie algebra }

\newcommand{\fm}{(\cdot,\cdot)}
\newcommand{\fh}{\mathfrak{h}}

\newcommand{\R}{\mathbb{R} }

\newcommand{\Z}{\mathbb{Z} }

\newcommand{\CA}{\mathcal{A} }

\newcommand{\CV}{\mathcal{V}}
\newcommand{\fg}{\mathfrak{g}}

\newcommand{\pa}{{\pi(\alpha)}}
\newcommand{\ep}{\hfill$\Box$}

\def\ad{\hbox{ad}}
\def\andd{\quad\hbox{and}\quad}
\def\sg{\sigma}
\def\a{\alpha}
\def\b{\beta}
\def\lam{\lambda}
\def\Lam{\Lambda}
\def\ep{\epsilon}
\def\andd{\quad\hbox{and}\quad}

\def\id{\hbox{id}}

\def\End{\hbox{End}}

\def\St{\mathfrak{St}}
\def\g{\mathfrak{g}}

\def\SL{\text{SL}_{2}}

\def\mod{\hbox{mod}}
\def\andd{\quad\hbox{and}\quad}
\def\ind{\hbox{ind}}

\def\v{{\mathcal V}}

\def\fm{(\cdot,\cdot)}
\def\a{\alpha}

\def\w{{\mathcal W}}

\def\sub{\subseteq}
\def\rd{\dot{R}}

\def\lam{\lambda}
\def\Lam{\Lambda}

\def\1k{\frac{1}{k}}

\def\la{\langle}
\def\ra{\rangle}
\def\rds{\dot{R}_{sh}}
\def\rdl{\dot{R}_{lg}}

\def\GL{GL}

\def\d{\delta}

\def\b{\beta}

\def\qed{\hfill$\Box$}

\def\sg{\sigma}

\def\sg{\sigma}

\usepackage{verbatim} 

\def\quadd{\quad\quad}

\def\ad{\hbox{ad}}

\def\bbbc{{\mathbb C}}
\def\bbbz{{\mathbb Z}}
\def\Z{{\mathbb Z}}

\def\bbbr{{\mathbb R}}
\def\bbbf{{\mathbb F}}

\def\bbbk{{\mathbb K}}

\def\xx{\mathcal X}

\def\aa{\mathcal A}

\def\ep{\epsilon}

\def\ll{{\mathcal G }}

\def\ll{\mathcal L}

\def\bb{{\mathcal B}}

\def\proof{{\noindent\bf Proof. }}

\def\rds{\dot{R}_{sh}}
\def\rdl{\dot{R}_{lg}}
\def\rde{\dot{R}_{ex}}
\def\St{\mathfrak{St}}
\def\g{\mathfrak{g}}

\def\SL{\text{SL}_{2}}
\def\span{\hbox{span}}

\def\te{\theta}

\def\DynkinNodeSize{1.5mm}
\def\DynkinArrowLength{2mm}
\tikzset{
	dnode/.style={
		circle,
		inner sep=0pt,
		minimum size=\DynkinNodeSize,
		fill=white,
		draw},
	middlearrow/.style={
		decoration={markings,
			mark=at position 0.8 with
			{\draw (0:0mm) -- +(+140:\DynkinArrowLength); \draw (0:0mm) -- +(-140:\DynkinArrowLength);},
		},
		postaction={decorate}
	},
	leftrightarrow/.style={
		decoration={markings,
			mark=at position 0.999 with
			{
				\draw (0:0mm) -- +(+135:\DynkinArrowLength); \draw (0:0mm) -- +(-135:\DynkinArrowLength);
			},
			mark=at position 0.001 with
			{
				\draw (0:0mm) -- +(+45:\DynkinArrowLength); \draw (0:0mm) -- +(-45:\DynkinArrowLength);
			},
		},
		postaction={decorate}
	},
	sedge/.style={
	},
	dedge/.style={
		middlearrow,
		double distance=0.6mm,
	},
	tedge/.style={
		middlearrow,
		double distance=1.0mm+\pgflinewidth,
		postaction={draw}, 
	},
	infedge/.style={
		leftrightarrow,
		double distance=0.5mm,
	},
}

\begin{document}

%
%
\title{Integral structures in extended affine Lie algebras}

\author{Saeid Azam, Amir Farahmand Parsa, Mehdi Izadi Farhadi}
\address
{Department of Mathematics\\ University of Isfahan\\Isfahan, Iran,
	P.O.Box: 81745-163 and\\
	School of Mathematics, Institute for
	Research in Fundamental Sciences (IPM), P.O. Box: 19395-5746.
} \email{azam@ipm.ir}
\address{Institute for
	Research in Fundamental Sciences (IPM), P.O. Box: 19395-5746.}\email{a.parsa@ipm.ir}
\address
{Department of Mathematics\\ University of Isfahan\\Isfahan, Iran,
	P.O.Box: 81745-163.} \email{m.izadi@sci.ui.ac.ir}
\thanks{This research was in part supported by
	a grant from IPM and carried out in
	IPM-Isfahan Branch.}


\begin{abstract}
We {construct} certain integral structures for the cores of reduced tame extended affine Lie algebras of rank at least 2.\ {One of the main tools to achieve this is a generalization of Chevalley automorphisms in the context of extended affine Lie algebras.}\ As an application, groups of extended affine Lie type associated to the adjoint representation are defined over arbitrary fields.
\end{abstract}
\maketitle

\section{Introduction}
\label{intro}\setcounter{equation}{0}

K. Saito introduced {\it extended affine root systems} (EARS) in 1985 to model the universal deformation of a simple elliptic singularity (see \cite{Sai85}).\ {Then in 1990, H\o egh-Krohn and  Torresani constructed a certain class of {complex} Lie algebras called {\it quasi-simple Lie algebras} whose roots fall into the class of extended affine root systems} (see \cite{HoTor90}).\ {These algebras are now known as {\it extended affine Lie algebras} (EALAs) due to their root systems.}\ {EALAs} and their root systems were axiomatically defined and thoroughly studied in \cite{AABGP97}.\

{It is straightforward to define EALAs over fields of characteristic 0 by the same axioms {(see \cite{Ne04})}, however, for positive characteristics this requires more subtle considerations.}\ {Recall that the core of an EALA is its subalgebra generated by non-isotropic root spaces.}\ By analogy with the arguments in classical Lie theory, here we introduce certain integral forms {for the cores of EALAs}, by means of which, we {are able to define a counterpart for these algebras} over fields of positive characteristic through reducing structure constants and then extending scalars.\ 

In this paper, we consider the {cores} of reduced tame extended affine Lie algebras of rank at least 2. To construct our desired integral form for such an algebra, we present an appropriate {spanning set which} is obtained from a Chevalley system (Definition \ref{chevalley-sys}). {To achieve this, the procedure for isotropic and non-isotropic root spaces is different. For non-isotropic root spaces,}  we use arguments that are mostly local by means of nilpotent pairs and their corresponding rank $2$ {finite dimensional simple} local subalgebras, together with certain local affine subalgebras. Then {following the same approach  as in classical Lie theory, to make {our} arguments global in a compatible way, {we employ the concept of a Chevalley automorphism (Definition \ref{chev1})}}. 

{{For the isotropic} root spaces the situation is more subtle} due to the fact that our knowledge of these spaces is limited. We overcome this by considering the concept of a reflectable base (see Definition \ref{ref0}) and analyzing the action of inner automorphisms based on root vectors corresponding to the ambient reflectable base. We emphasize that our integral forms enjoy the very important property of being invariant under the action of the associated extended affine Weyl group which is, for instance, essential for constructing the corresponding groups of extended affine Lie type over arbitrary fields.

In \cite{AP19} we produced groups of extended affine Lie type over complex numbers associated with the tame reduced extended affine Lie algebras generalizing the results in \cite{Kry95} from simply-laced to arbitrary type. In this direction, there is also a work by Morita and Sakaguchi \cite{MS06}. Here, as an application of integral forms, we are able to define groups of extended affine Lie type over arbitrary fields. Then we show that when the underlying field is finite, such groups under {certain} mild conditions are finitely generated; {this generalizes the concept of arithmetic groups in non-Archimedean semisimple groups} (see, e.g., \cite{zbMATH05707462}).

Integral forms, in turn, are very important objects in {modular form theory.} Further, they appear in the modular moonshine program of Borcherds and Ryba {(see e.g., \cite{zbMATH00841167,zbMATH01425114,zbMATH00912132} and for more recent works, see \cite{zbMATH06377286,zbMATH06436297}).\ With these integral forms in hand, one might expect the emergence of} a modular form theory in the context of EALAs.\ Moreover, integral forms for EALAs can furnish the ground for defining integral cohomology of extended affine Lie algebras {(see e.g., \cite{zbMATH04091712,zbMATH05208125}).}

The structure of the paper is as follows.\ In Section~\ref{sub-sec-EALA} basic definitions and notation from the theory of EALAs are provided.\ Section~\ref{Chevalley automorphism and extended affinization} deals with the concept of extended affinization where we introduce certain Chevalley automorphisms and $\sg$-Chevalley pairs for extended affine Lie algebras in Definition \ref{chev1}.\ Then we show that a $\sg$-Chevalley pair for an EALA induces a Chevalley automorphism on its affinization (see Proposition \ref{pro115} below) and then {apply} it to construct explicit Chevalley automorphisms for toroidal Lie algebras (Corollary \ref{cor115-0}), affine Lie algebras (Corollary \ref{cor115}) and elliptic Lie algebras (Corollary \ref{cor115-2}), starting from a Chevalley automorphism of a finite dimensional simple Lie algebra {or an affine Lie algebra}.\ {Section~\ref{Preliminary results on integral structure} is an auxiliary section to Section~\ref{Integral structure of the core} where we provide some preliminary results 
which are {higher nullity extensions} of properties of integral forms for finite and affine Lie algebras.} {In particular we define the notion of a Chevalley system (Definition \ref{chevalley-sys}) {for an extended affine Lie algebra.}} In Section~\ref{Integral structure of the core}, {the integral {structure} of the core of an EALA equipped with a Chevalley automorphism is studied and via {a Chevalley system and} a reflectable base, an integral form for the core is introduced} (see Theorems \ref{thmnew5} and \ref{thmnew1}).\ Section~\ref{On the Uniqueness of integral structure} is to show the independence of our integral forms from the choice of the {Chevalley system} (see Theorem \ref{simply}).\ In Section~\ref{chevalley basis and multi-loop realization}, multi-loop algebras based on extended affine Lie algebras and the algebra of Laurent polynomials are studied.\ It is shown in Theorem \ref{thm12} that with certain conditions on the involved automorphisms, the affinization procedure, when applied to {a} multi-loop algebra leads to an extended affine Lie algebra. We also investigate how a Chevalley automorphism on the underlying extended affine Lie algebra can be lifted to a Chevalley automorphism on the multi-loop affinization through Proposition \ref{lem12}.\ Finally in Section~\ref{sec:appl}, we present two rather immediate applications of integral forms.\ First these integral forms are used to define extended affine Lie algebras over fields of positive characteristic (see Definition \ref{defK-form}) then define certain groups associated with them.\ It is shown in Theorem \ref{thmgealtf} that the adjoint forms of groups of extended affine Lie type are always infinite and that over finite fields are finitely generated.

{Since our arguments for integral forms are essentially correspond to local rank $2$ subalgebras, we have excluded type $A_1$
which seems to require a different approach; we {will} address this type in a separate work.
We also emphasize that {since} our main motive for the present work is to define groups of extended affine Lie type over an arbitrary field, and on the group level EALAs and their cores have isomorphic associated groups} (see \cite{AP19}),  we have restricted our study in this work to the cores.
We conclude the introduction by emphasizing that
one might be able to extend the integral {forms} from the core to the whole EALA with techniques presented in \cite{Ne04}; this also will be addressed in {our} future works.

\section{Extended affine Lie algebras}\label{sub-sec-EALA}
\setcounter{equation}{0}
In this section we provide some basic concepts, definitions, and notations which will be used in the sequel. In particular we recall definitions of extended affine Lie algebras and root systems.
We also give some basic results on extended affine theory.  

All vector spaces and algebras are assumed to be over field of complex numbers $\bbbc$ unless
otherwise stated. By $\fh^\star$, we denote the dual spaces of a vector space $\fh$ and by $\la T\ra$, we denote the subgroup generated by a subset $T$ of the ground group or vector space.\

Let $\fg$  be a Lie algebra, $\fh$ be a non-trivial
subalgebra of $\fg$ and $\fm$ be a symmetric bilinear form on $\fg$. The triple $(\fg,\fm,\fh)$ is called an {\it extended affine Lie algebra} if the following five axioms hold:

(A1) $\fm$ is invariant and non-degenerate on $\fg$,

(A2) $\fh$ is a finite-dimensional Cartan subalgebra of $\fg$.

Axiom (A2) means that 
$$\fg=\sum_{\a\in\fh^\star}\fg_\a\hbox{ where }\fg_\a=\{x\in\fg\mid [h,x]=\a(h)x\hbox{ for all }h\in\fh\},
\hbox{ and }\fg_0=\fh.$$ 
Let $R$ be the set of roots of $\fg$, namely $R=\{\a\in\fh^\star\mid\fg_\a\neq\{0\}\}$.
It follows from (A1)-(A2) that the form $\fm$ restricted to $\fh$ is non-degenerate and so it can be transferred
to $\fh^\star$ by $(\a,\b):=(t_\a,t_\b)$ where $t_\a\in\fh$
is the unique element satisfying $\a(h)=(h,t_\a)$, $h\in\fh$.
Let  $\v:=\span_{\bbbr}R$ and denote the radical of the form on $\v$ by $\v^0$. We set
$$R^0:=\{\a\in R\mid (\a,\a)=0\}\andd R^\times:=R\setminus R^0.$$
Then $R=R^0\uplus R^\times$ is regarded as the decomposition of roots into
{\it isotropic} and {\it non-isotropic} roots, respectively. 

(A3) For $\a\in R^\times$ and $x\in\fg_\a$, $\ad(x)$ acts locally nilpotently on $\fg$.

(A4) The $\bbbz$-span of $R$ in $\fh^\star$ is a free abelian group of rank $\dim\v$.

(A5) (a) $R^\times$ is indecomposable,

\quad\quad\;\;(b) $R^0$ is {\it non-isolated}, meaning that $R^0=(R^\times-R^\times)\cap\v^0.$

In lemma below,  we gather some basic facts about extended affine Lie algebras and root systems  which will be needed in the sequel.
For parts (i)-(iii) see \cite[Chapter I]{AABGP97} and for part (iv) see
\cite[Remark 1.5(ii)]{Az06}. We frequently use this lemma, often without further reference.
\begin{lem}\label{oldlem}
Let $(\fg,\fm,\fh)$ be and extended affine Lie algebra with root system
$R$. Let $\a\in R^\times$ and $\b\in R$.

(i) $\dim(\fg_\a)=1$.

(ii) $[x_\b,x_{-\b}]=(x_\b,x_{-\b}) t_\b,$ $ x_{\pm\b}\in\fg_{\pm\b}.$

(iii) $[\fg_\b,\fg_{-\b}]=\bbbc t_\b.$

(iv) $[\fg_\a,\fg_\b]\not=\{0\}$ if $\a+\b\in R$. In particular, if $\a+\b\in R^\times$,
$[\fg_\a,\fg_\b]=\fg_{\a+\b}$.
\end{lem}

Let $(\fg,\fm,\fh)$ be an EALA with root system $R$.\
For $\a\in R^\times$, we set $\a^\vee:=2\a/(\a,\a)$.
The reflection $w_\a\in\GL(\fh^\star)$ is then defined by
$w_\a(\b):=\b-(\b,\a^\vee)\a$, $\b\in\fh^\star$.
Let $\w_\fg$ be the Weyl group of $\fg$; the subgroup of $\GL(\fh^\star)$ generated by reflections $w_\a$, $\a\in R^\times$.\ By \cite[Theorem I.2.16]{AABGP97}, the root system $R$ of $\fg$  is an irreducible extended affine root system in the following sense.

\begin{DEF}\label{def2}\rm{
		Let $\v$ be a finite-dimensional real vector space equipped with a non-trivial positive semidefinite symmetric bilinear form $\fm$, and $R$ a subset of $\v$. Then the {triple} $(\v,\fm,R)$, or $R$ if there is no confusion, is called an {\it extended affine root system} (EARS)
		{if the following axioms hold:}}
	{\begin{itemize}
			\item[(R1)] $\la R\ra$ is a full lattice in $\v$,
			\item[(R2)] $(\b,\a^\vee)\in\bbbz$, $\a,\b\in R^\times$,
			\item[(R3)] $ w_\a(\b)\in R$ for $\a\in R^\times$, $\b\in R$,			
			\item[(R4)] $R^0= \v^0\cap(R^\times - R^\times)$,
			\item [(R5)] $\a\in R^\times\Rightarrow 2\a\not\in R$.
		\end{itemize}
We say that $R$ is {\it irreducible} or
		{\it connected} if  $R^\times$ cannot be written as the union of two of its non-empty orthogonal subsets.}
\end{DEF}		
\begin{rem}\label{rem100}{\rm
		One checks that the definition of an irreducible EARS given above coincides with the definition of an extended affine root system  given 
		in \cite[Definition II.2.1]{AABGP97}.}
\end{rem}

We next recall from \cite[Chapter II]{AABGP97} some facts about the structure of an EARS.
Let $(\v,\fm,R)$ be an {irreducible} EARS.\ It follows that the canonical image $\bar R$ of $R$ in $\bar\v:=\v/\v^0$ is a finite root system
in $\bar\v$. The {\it type} and the {\it rank} of $R$ is defined to be the type and the rank of $\bar R$, {and the dimension of $\v^0$ is called the {\it nullity} of $R$.}
The EARS $R$ is called {\it reduced} if the finite root system $\bar R$ is reduced.

We fix a root base $\{\bar\a_1,\ldots,\bar\a_\ell\}$ for $\bar R$. Considering the canonical map $\CV\rightarrow\bar\CV$, and for each $i$, we fix a preimage ${\dot\a}_i$ of $\bar\a_i$ in $R$.
Let $\dot\CV:=\hbox{Span}_\R\{\dot\a_1,\ldots,\dot\a_\ell\}$.
Then $\CV=\dot\CV\oplus\CV^0$ and 
$\dot R:=\{\dot\a\in\dot\CV\mid\dot\a+\sg\in R\hbox{ for some }\sg\in \CV^0\}$ is a finite root system in $\dot\CV$ isomorphic to $\bar R$.  Moreover, one obtains a
description of $R$ in the form
\begin{equation}\label{eq99}
R=R(\rd,S,L,E):=(S+S)\cup (\rds+S)\cup (\rdl+L)\cup (\rde+E),
\end{equation}
where $\rds$, $\rdl$ and $\rde$ are the sets of short, long and extra long roots of $\rd$, respectively, and $S$, $L$ and $E$ are
certain subsets of $R^0$, called {\it (translated) semilattices}, which interact in a prescribed way (see \cite[Chapter II]{AABGP97} for details).
In particular $\la L\ra\sub S$ and $\la S\ra/\la L\ra$ is a vector space over field of two elements whose dimension is called the
{\it twist number} of $R$.    {If $R$ is reduced, then in (\ref{eq99}), $\rd_{ex}$ and $E$ are interpreted as empty sets, so
\begin{equation}\label{eq99-1}
R=R(\rd,S,L)=(S+S)\cup (\rds+S)\cup (\rdl+L).
\end{equation}
for some  reduced finite root system $\rd$}. 
Let $\Lam:=\la R^0\ra$, then $\Lam=\la S\ra$ and $\Lam$ is a lattice in $\v^0=span_{\bbbr}R^0$. Moreover,  $S$ contains a $\bbbz$-basis of $\Lam$, namely
\begin{equation}\label{sen1}
\Lam=\bbbz\sg_1\oplus\cdots\oplus\bbbz\sg_\nu,\hbox{ for some }\sg_1,\ldots,\sg_\nu\in S.
\end{equation} 
 It follows that  $\sg_1,\ldots,\sg_\nu$ can be chosen such that
\begin{equation}\label{ned1}
\begin{array}{l}
\la L\ra=k\Lam_1\oplus\Lam_2,\\
\Lam_1=k\bbbz\sg_1\oplus\cdots\oplus k\bbbz\sg_t,\\
\Lam_2=\bbbz\sg_{t+1}\oplus\cdots\oplus\bbbz\sg_\nu,
\end{array}
\end{equation}
where $t$ is the twist number of $R$, $k=3$ if $X=G_2$ and $k=2$ otherwise.
Finally, we recall that if $R$ is non-simply laced then $S=S+L$ and $L=kS+L$.

	For later use, we normalize the form such that
\begin{equation}\label{fm1}
(\a,\a)=2\quad\hbox{for}\quad\a\in\rds.
\end{equation}

For $\a\in R^\times$, define the reflection $w_\a\in\GL(\v)$ by $w_\a(\b)=\b-(\b,\a^\vee)\a$. Then $\w_R:=\langle w_\a\mid\a\in R^\times\rangle\sub\hbox{GL}(\CV)$ is called the {\it Weyl group} of $R$. We note that the map $\bar{\;}$ induces an epimorphism $\w_R\longrightarrow \w_{\bar R}$, where $\w_{\bar R}$ is the Weyl group of $\bar R$.

The action of Weyl group on the root system results in some important objects known as ``reflectable bases''. 
\begin{DEF}\label{ref0}
A subset $\Pi$ of $R^\times$ is called a {\it reflectable base} for $R$ if it is minimal with respect to the property that each non-isotropic root can be recovered by reflecting roots of $\Pi$ with respect to hyperplanes determined by elements of $\Pi$. In other words, 
if $\w_\Pi$ is the subgroup of $\w$ generated by reflections based on $\Pi$, then $\w_\Pi\Pi=R^\times$ and no proper subset of $\Pi$ has this property. 
\end{DEF}
Reflectable bases resemble in  a weaker sense the concept of root bases for finite and affine root systems. For a characterization of reflectable bases see \cite{AYY12}, \cite{ASTY19}. In Section \ref{Integral structure of the core},  we introduce for each type a particular reflectable base (see Table \ref{tab11}) which plays a significant role in our understanding of the integral structure of extended affine Lie algebras.

The following lemma discusses the choice of the finite root system $\dot R$ in
 (\ref{eq99}). In fact, one notes that  $\dot R$ depends on the choice of the pre-images $\dot\a_i$ of $\bar\a_i$ for $1\leq i\leq \ell$. The following result shows that $\dot R$ can be chosen such that a prescribed property is satisfied. 
\begin{lem}\label{isos}
For a fixed $\sg\in R^0$, the finite root system $\dot R$ in (\ref{eq99}) can be chosen in such a way that $\dot\a+\sg\in R$ for some $\dot\a\in\dot R$. In particular, with this choice of ${\dot R}$,  we have $\sg\in S$.
\end{lem} 

\proof
Let $\sg\in R^0$. By (R4), there exists $\a\in R^\times$ such that $\a+\sg\in R$. Let $\gamma=\a+\sg$. Then
$\bar\a=\overline{\gamma-\sg}\in\bar R$. So there exists $w\in\w$, such that $\bar{w}(\bar\a)=\bar{\a}_{i_0}$ for some $1\leq i_0\leq \ell$.
Set $\dot{\a}_{i_0}:={w}(\a)\in R$. Then
${\dot \a}_{i_0}+\sg=w(\gamma-\sg)+\sg=w(\gamma)\in R$. 
So in the procedure of choosing pre-images for root bases $\bar\a_j$, described prior to (\ref{eq99-1}), we may choose $\dot\a_{i_0}=w(\a)$ as the fixed pre-image of $\bar{\a}_{i_0}$.  Then as $\dot{\a}_{i_0}\in\dot R$ and $\dot\a_{i_0}+\sg\in R$, we have  $ \sg\in S$.\qed

We conclude this section with a discussion on the core and centerless core of an extended affine Lie algebra $(\fg,\fm,\fh)$. 
The {\it core} of an EALA $\fg$ is by definition, the subalgebra $\fg_c$ of $\fg$ generated by non-isotropic root spaces.\
It follows that $\fg_c$ is a perfect ideal of $\fg$.
The EALA $\fg$ is called {\it tame} if $\fg_c^\perp:=\{x\in\fg\mid (x,\fg_c)=\{0\}\}={\mathcal Z}(\fg_c)$.

Since $\fg_c$ is an ideal of $\fg$, it is an $\fh$-module and so
$\fg_c=\sum_{\a\in R}\fg_c\cap\fg_\a$, where $\fg_c\cap\fg_\a=\fg_\a$ if $\a\in R^\times.$ If $\sg\in R^0$, then
by axiom (A5), there exists $\a\in R^\times$ with $\a+\sg\in R^\times$.
Then by Lemma \ref{oldlem}, $0\neq[\fg_{-\a},\fg_{\a+\sg}]\sub\fg_\sg$.
Thus $\fg_c\cap\fg_\a\neq 0$ for all $\a\in R$.
Considering this, and by abuse of language, we call $R$ the root system of $\fg_c$.

Assume further that $\fg$ is tame and set
 $\fg_{cc}:=\fg_c/{\mathcal Z}(\fg_c)$ where   ${\mathcal Z}(\fg_c)$ denotes the center of $\fg_c$. The algebra $\fg_{cc}$ is called  the {\it centerless core} of $\fg$. Since ${\mathcal Z}(\fg_c)=\fg_c^\perp$, the form on $\fg$ induces a form on
$\fg_{cc}$ which is non-degenerate. Moreover, we have
\begin{equation}\label{core1}
\fg_{cc}=\sum_{\a\in R}(\fg_{cc})_\a
\quad\hbox{where}\quad
(\fg_{cc})_\a=\frac{\fg_c\cap\fg_\a+{\mathcal Z}(\fg_c)}{{\mathcal Z}(\fg_c)}.
\end{equation} 

\begin{lem}\label{zcore}
(i) Suppose $x_{\a+\sg}\in\fg_{\a+\sg}$, $x_{-\a}\in\fg_{-\a}$,
 $\a\in R^\times$ and $\sg\in R^0$. If $x=[x_{\a+\sg},x_{-\a}]\not=0$, then $x\not\in{\mathcal Z}(\fg_c)$.

(ii) For $\a\in R^\times$ and $\sg\in R^0$
For each $\a\in R$, $(\fg_{cc})_\a\not=0$.
\end{lem}

\proof (i) By Lemma \ref{oldlem}(iv), there exist $x_{-\a-\sg}\in\fg_{-\a-\sg}$ and $x_{\a}\in\fg_{\a}$
 such that
 $0\neq y:=[x_{-\a-\sg},x_{\a}]\in\fg_c\cap\fg_{-\sg}.$
 By Lemma \ref{oldlem}(iii), we may assume that  $[x_\a,x_{-\a}]=t_\a$ and
 $[x_{\a+\sg},x_{-\a-\sg}]=t_{\a+\sg}.$
 Then
 \begin{eqnarray*}
 (x,y)&=&([x_{\a+\sg},x_{-\a}],[x_{-\a-\sg},x_{\a}])\\
 &=&([x_{\a+\sg},[x_{-\a},[x_{-\a-\sg},x_{\a}]])\\
&=&-(x_{\a+\sg},[x_{-\a-\sg},[x_\a,x_{-\a}]])\\
&=&-(x_{\a+\sg},[x_{-\a-\sg},t_\a])\\
&=&-([x_{\a+\sg},x_{-\a-\sg}],t_\a)\\
&=&-(t_{\a+\sg},,t_\a)=-(\a,\a)\not=0.
%
 \end{eqnarray*}
 This shows that $x\not\in\fg_c^\perp={\mathcal Z}(\fg_c)$.
 Thus $x\in (\fg_c\cap\fg_\sg)\setminus {\mathcal Z}(\fg_c)$.

(ii) We must show that for each $\a\in R$, $\fg_c\cap\fg_\a\not\sub
{\mathcal Z}(\fg_c)$. If $\a\in R^\times$, then by  \cite[Proposition 1.4(ii)]{Az06}, ${\mathcal Z}(\fg_c)\cap\fg_\a=0$ and we are done.
 Assume next that $\sg\in R^0$. From axiom (A5), there exists $\a\in R^\times$ with $\pm(\a+\sg)\in R$. 
 Now by Lemma \ref{oldlem}(iv), there exists $x_{\a+\sg}\in\fg_{\\a+\sg}$ and  $x_{-\a}\in\fg_{\pm\a}$
 such that
  $0\neq x:=[x_{\a+\sg},x_{-\a}]\in\fg_c\cap\fg_{\sg}$.
  Then the required claim follows from part (i).\qed

Considering (\ref{core1}) and Lemma \ref{zcore}, we say, by abuse of language, that $R$ is the root system of $\fg_{cc}$. In other words, when we say $R$ is the root system of $\fg_c$ or $\fg_{cc}$, we actually mean that $R$ is the set of weights of $\fg_c$ or $\fg_{cc}$ as an $\fh$-module.

\section{Chevalley automorphism and extended affinization}\label{Chevalley automorphism and extended affinization}\setcounter{equation}{0}
In this section, we first introduce the concept of a loop algebra of an extended affine Lie algebra $\fg$ equipped with a finite order automorphism $\sg$, based on a quantum torus. We use this to construct a new extended affine Lie algebra out of $\fg$, called the extended affinization of $\fg$.  We introduce the concepts of Chevalley automorphisms and $\sg$-Chevalley pairs for extended affine Lie algebras and show that a $\sg$-Chevalley pair for $\fg$ induces a Chevalley automorphism on its affinization. This will be used then to construct explicit Chevalley automorphisms for toroidal Lie algebras, affine Lie algebras and elliptic Lie algebras, starting from a Chevalley {automorphism} of a finite dimensional simple Lie algebra.

Throughout this section {we assume that  $\CA$ is a quantum torus based on  a  free abelian group $\Lam$ of finite rank.} More precisely, {$\aa$ has a $\bbbc$-basis $\{a^\lam\mid\lam\in\Lam\}$} satisfying 
\begin{equation}\label{coordinate}
\aa=\sum_{\lam\in\Lam}^\oplus\aa^\lam\quad\hbox{with}\quad\aa^\lam=\bbbc a^\lam\andd 
a^\lam\cdot a^\gamma=\mu(\lam,\gamma)a^{\lam+\gamma}
\end{equation}
where $\mu:\Lam\times\Lam\rightarrow \bbbc^\times$ is a $2$-cocycle satisfying
$$\mu(\lam,\tau)=\mu(\tau,\lam)=\mu(-\lam,-\tau)\andd
\mu(0,0)=1.$$
We consider the $\Lambda$-graded  invariant non-degenerate   {symmetric}
bilinear form  $\ep:\aa\times\aa\rightarrow\bbbc,$ given by
$$\ep(a^\lam,a^\tau)=\mu(\lam,\tau)\d_{\lam,-\tau}\qquad(\lam,\tau\in\Lam).
$$
(Here $\d_{\sg,-\tau}$ denotes the Kronecker Delta.)

Let $(\fg,\fm, \fh)$ be an extended affine Lie algebra with root system $R$.
Let $\sg$ be an automorphism of $\fg$ of period $m$, and $\omega:=e^{2\pi i/m}.$
Then $\sg$ induces a $\bbbz_m$-grading
$\fg=\sum_{\bar i\in\Z_m}\fg^{\bar i}$
where
$$\fg^{\bar i}=\{x\in\fg\mid \sg(x)=\omega^ix\}.
$$
We note that $\fg^{\bar j}=\pi_j(\fg)$, where
\begin{equation}\label{eqram}
\pi_j=\frac{1}{m}\sum_{i=0}^{m-1}\omega^{-ij}\sg^i.
\end{equation}
(The scalar $1/m$ is inserted for some normalization on the level of roots.) 
Moreover,
\begin{equation}\label{eq11}
\sg\pi_j=\pi_j\sg=\omega^{-j}\pi_j\andd [\pi_{i}(x),\pi_j(y)]=\pi_{i+j}[x,\pi_j(y)],
\end{equation}
for $x,y\in\fg$ (see \cite[Section 3]{AHY13}).  We set $\pi:=\pi_0$.
Note that
$\pi(\fg)=\fg^{\bar 0}=\fg^\sg$, where $\fg^\sg$ is the set of fixed points of $\fg$, under $\sg$. 

{Assume further that $\sg(\fh)=\fh$. Then we also have
{$\pi(\fh)=\fh^0=\fh^\sg$. 
The automorphism $\sg$ also induces an automorphism, denoted again by $\sg$, on $\fh^*$ by
$\sg(\a)(h)=\a(\sg^{-1}(h))$, $\a\in \fh^\star$, $h\in\fh$. As above
we denote the map $\pi_0\in\End(\fh^\star)$ by $\pi$. One observes that for $\a\in\fh^\star$, $\pi(\a)$ can be identified with $\a_{|_{\fh^\sg}}$. 
It is easy to check that
\begin{equation}\label{eqram3}
(\pi(\a),\pi(\b))=(\a,\pi(\b))\quad\hbox{for}\quad \a,\b\in\fh^\star.
\end{equation}
We now set
$$\pi(R):=\{\pi(\a)\mid \a\in R\}.$$}
Recall from Section \ref{sub-sec-EALA} that the form on $\fh$ transfers to a form on $\fh^\star$.}

Throughout this section, we assume that $\sg$ satisfies
\begin{equation}\label{eq1}
\begin{array}{l}
\sg^m=\id,\\
 \sg(\fh)=\fh,\\
  (\sg(x), \sg(y))=(x, y)\hbox{ for all }x, y\in\fg,\\
 C_{\fg^\sg}(\fh^\sg)\sub\fh^\sg,\\
 (\a,\pi(\a))\not=0,\hbox{ for some }\a\in R. 
\end{array}
\end{equation}
While the first four conditions given above looks natural for people familiar with the passage from finite dimensional theory to affine theory, the last less familiar condition is inserted in order to guarantee the existence of some non-isotropic vector $\pi(\a)$.   

Next, we  consider the Lie algebra $\fg \otimes \CA$ with bracket defined by
\[ [x \otimes a, y \otimes b ] := [x,y] \otimes {\mu(\lam,\tau)}ab\]
for $x,y \in \fg$, $ a\in\aa^\lam,\;b\in\aa^\tau$.
Now we define a form on $\fg\otimes \CA$ {by linear extension} of
\begin{equation}\label{newtempeq8}
(x \otimes a, y \otimes b ) = (x,y) \epsilon(a,b),
\end{equation}
for  $x,y \in \fg$ and $ a,b \in \CA$.  It is easy to check that this is  a $\Lam$-graded invariant symmetric  bilinear form on ${\fg} \otimes \CA$.

\begin{DEF}\label{loopalgebra}
	Let $\rho:\Lam\rightarrow\Z_m$ be a group epimorphism and set $\bar\lam:=\rho(\lam)$ for $\lam\in\Lam$.  The subalgebra
	\[ \tilde\fg := L_\rho(\fg, \CA):= \bigoplus_{\lambda \in \Lambda} (\fg^{\bar{\lambda}} \otimes \CA^\lambda)\]
	of $\fg \otimes \CA$  is called the {\it loop algebra} of $\fg$
	relative to $\rho$ and $\CA$. In the case that $\sg=\id$ and $\rho=0,$ we denote
	$L_\rho(\fg,\CA)$ by
	$L(\fg,\CA)$ and  note that  $L(\fg,\CA)=\fg\otimes\CA$. {The Lie algebra $L(\fg,\CA)$ is refereed to as a \it{toroidal} Lie algebra.}
\end{DEF}
From definition, it is clear that $\tilde{\fg}$  is a
$\Lambda$-graded Lie algebra with homogeneous spaces
$\tilde{\fg}^\lambda:=\fg^{\bar\lambda}\otimes\CA^\lambda$,
$\lambda\in\Lambda$.

The following result is standard, see for example \cite[Lemma 7.5]{AHY13}.

\begin{lem}\label{tildefgform}
	The form on $\fg \otimes \CA$ restricted to $\tilde{\fg}$ is a
	$\Lam$-graded invariant non-degenerate symmetric  bilinear form.
\end{lem}

For $\a\in\fh^\star$, let $\fg_{\pi(\a)}$ be the sum of all root spaces $\fg_\b$ with $\pi(\b)=\pi(\a)$. For $\lam\in\Lam$, let $\fg^{\bar\lam}_{\pi(\a)}=\fg_{\pi(\lam)}\cap\fg^{\bar\lam}$. Then we have
\begin{equation}\label{gg}
\tilde{\fg}^\lambda = \fg^{\bar{\lambda}} \otimes \CA^\lambda = \bigoplus_{\pi(\alpha) \in \pi(R)}
( \fg^{\bar \lambda}_{\pi(\alpha)} \otimes \CA^\lambda ).
\end{equation}

Next, we set $\tilde\fh:=\fh^\sg\otimes 1$, and we embed $(\fh^\sg)^\star$ in $\tilde\fh^\star$.
The adjoint action of $\tilde{\fh}$ on $\tilde{\fg}$ stabilizes the spaces $\fg^{\bar{\lambda}} \otimes \CA^\lambda$, $\lam\in\Lam$. Let
$\tilde\fg_{\pi(\a)}=\{x\in\tilde\fg\mid [x,h]=\pi(\a)(h)x\hbox{ for all }h\in\tilde\fh\}$. It then follows that
\begin{equation}\label{gtilderootdec}
\tilde\fg = \bigoplus_{\pi(\alpha) \in \pi(R)} \tilde\fg_{\pi(\alpha)}.
\end{equation}
Moreover, we have
\begin{equation}\label{tildefgpa}
{\tilde\fg}_{\pi(\alpha)} = \bigoplus_{\lambda \in \Lambda} ({\fg}^{\bar{\lambda}}_{\pi(\alpha)}\otimes\CA^\lam),
\end{equation}
namely we have a compatible $(\langle\pi(R)\rangle, \Lambda)$-grading
\begin{equation}
\tilde{\fg} = \bigoplus_{\lambda \in \Lambda, \gamma \in\langle\pi(R)\rangle} \fg^{\bar \lambda}_\gamma \otimes \CA^\lambda,
\end{equation}
where  for ${\lambda \in \Lambda}$, $\tilde\fg_\gamma^\lambda =\{0\}$ if $\gamma \notin \pi(R)$.

Next, $\Lam$ being  free abelian allows us to embed it in the $\bbbc$-vector space
$\CV:=\bbbc \otimes_\Z \Lambda.$ We set
\begin{equation}\label{hatfg}
\hat{L}_\rho(\fg,\CA)=\hat \fg:=
\tilde \fg\oplus \CV \oplus \CV^\star
\andd \hat\fh:=\tilde \fh\oplus \CV\oplus \CV^\star=(\fh^0\otimes 1)\oplus\CV\oplus\CV^\star,
\end{equation}
and make $\hat\fg$ into a Lie algebra with the Lie bracket:
\begin{equation}\label{bracketloop}\begin{array}{l}
\;[d,x]=d(\lambda) x,  \quad d\in\CV^\star,\;x\in \tilde \fg^\lambda,\;\lambda \in \Lambda,\\
\;[\CV,\hat\fg]=\{0\},\\
\;[x,y]=[x,y]_{\tilde\fg}+(x,y)\lambda,\;x\in\tilde\fg^\lam,\;y\in \tilde\fg, \end{array}\end{equation} where by
$[\cdot,\cdot]_{\tilde\fg}$ and $(\cdot,\cdot)$, we mean the Lie
bracket and the  bilinear form on ${\tilde\fg},$ respectively.
Now the form on
$\tilde\fg$ extends
to a non-degenerate form on $\hat \fg$ by
\begin{equation}\label{formloop}
\begin{array}{l}
(\CV,  \CV)=(\CV^\star,\CV^\star)=(\CV,\tilde\fg)=(\CV^\star,\tilde \fg):=\{0\},\\
(v,d)=(d,v):=d(v),\quad  d\in \CV^\star, v\in \CV.
\end{array}
\end{equation}

\begin{DEF}We call the Lie algebra $(\hat{L}_\rho(\fg,\aa),\fm,\hat\fh)$ an extended affinization of $(\fg,\fm,\fh)$.
\end{DEF}

By imbedding  $\Lam\hookrightarrow\CV \hookrightarrow(\CV^\star)^\star$ via the natural paring, we may identify $\Lam$ inside ${\hat\fh}^\star$.    
We also identify  $(\fh^\sg)^\star$ inside $(\hat{\fh})^\star$. 
Then we have
\begin{equation}\label{newtempeq13}
\hat\fg=\bigoplus_{\hat\alpha\in\hat{R}}{\hat\fg}_{\hat\alpha},
\end{equation}
where
$${\hat\fg}_{\hat\alpha}:=\{x\in\hat\fg\mid [\hat h,x]=\hat\alpha(\hat h)x\hbox{ for all }\hat h\in \hat\fh\},
$$
and
$$\hat R=\{\hat\a\in(\hat{\fh})^\star\mid \hat{\fg}_{\hat\a}\ne 0\}\sub\pi(R)\oplus\Lam.$$
We have, for $\a\in R$ and $\lambda\in\Lam$,
\begin{equation}\label{hatfgpa+lambda}
\hat\fg_{\pa+\lambda}=\left\{\begin{array}{ll}
\fg_\pa^{\bar\lambda}\otimes\CA^\lam&\mbox{if }\pa+\lambda\not=0,\\
(\fg^0_{\pi(0)} \otimes \CA^0 )\oplus \CV \oplus \CV^\star&\hbox{if
}\pa+\lambda=0.
\end{array}\right.
\end{equation}
Finally, we mention that
\begin{equation}\label{234}
\hat R=\cup_{\lam\in\Lam}(\pi(R_{\bar\lam})+\lam),\quad{where}\quad
R_{\bar\lam}=\{\a\in R\mid \fg^{\bar\lam}_\a\not=\{0\}\}.
\end{equation}

\begin{thm}\label{thm09}
The extended affinization  $(\hat\g,\fm,\hat\fh)$ of $(\fg,\fm,\fh)$ is an extended affine Lie algebra with root system 
$\hat R$.
\end{thm}
\proof
The fifth condition of (\ref{eq1}) guarantees the existence of non-isotropic roots in $\hat R$. Then the result follows from \cite[Theorem 7.21]{AHY13}.

\begin{rem}\label{remt0}
{One observes that  $\tilde\fg$ is in fact the fixed point subalgebra of $\fg\otimes\CA$ under the automorphism $\sg\otimes\eta$ where $\eta$ is the automorphism of $\CA$ given by $\eta(a^\lam)=\omega^{-\bar\lam}a^\lam.$}
\end{rem}

In what follows we investigate how a Chevalley automorphism of an
extended affine Lie algebra $\fg$  can be promoted to an automorphism of its affinization $\hat\fg$. 
\begin{DEF}\label{chev1} Let $(\fg,\fm,\fh)$ be an extended affine Lie algebra and $\sg$ be an automorphism of $\fg$ satisfying (\ref{eq1}).

\noindent(i) We call an automorphism $\tau$ of $\fg$ a {\it Chevalley automorphism} if

- $\tau$ is of finite order,

- $\tau(h)=-h$ for $h\in\fh$.\\
\noindent(ii) A pair $(\tau,\psi)$ of finite order automorphisms of $\fg$ is called a $\sg$-Chevalley pair if

- $\tau$ is a Chevalley automorphism, 

- $\sg\tau=\tau\sg$,

- $\psi(\fg^{\bar j})=\fg^{-\bar j}$, for $j=0,\ldots,m-1$,

- $\psi$ preserves the form,

- $\psi(h)=h$ for $h\in\fh^{\bar{0}}$.
\end{DEF}

\begin{lem}\label{chev4}
Assume that $\tau$ is a Chevalley automorphism of $\fg$. Then

(i) {$\tau(\fg_\a)=\fg_{-\a}$ for $\a\in\fh^\star$.}

(ii) {$(\tau(x),\tau(y)=(x,y)$ for $x,y\in\fg$.}

(iii) $\tau$ is of even order.

\end{lem}

\proof
Since $\tau_{|_\fh}=-\id$, we have for $x\in\fg_\a$ and $h\in\fh$, 
$[h,\tau(x)]=\tau([-h,x])=-\a(h)\tau(x)$, so (i) holds.

Next, let $\a\not=0$, $x_\a\in\fg_\a$ and $x_\b\in\fg_{\b}$, then 
by (i), 
$(x_\a,x_\b)=(\tau(x_\a),\tau(x_\b))=0$ if $\a+\b\not=0$. 
If $\b=-\a,$ then
$$t_{-\a}(\tau(x_\a),\tau(x_{-\a}))=[\tau(x_\a),\tau(x_{-\a})]
=\tau([x_\a,x_{-\a}])=-t_\a(x_\a,x_{-\a}).$$
Thus $(\tau(x_\a),\tau(x_{-\a}))=(x_\a,x_{-\a})$. If $\a=0$, then for $h,h'\in\fg_0=\fh$ we have 
$(\tau(h),\tau(h'))=(-h,-h')=(h,h')$. All together, we get (ii). Part (iii) is clear from  (i).\qed

\begin{lem}\label{chev2}
{Let $(\tau,\psi)$ be $\sg$-Chevalley pair for $\fg$. Then
$\psi\tau(\fg_{\pi(\a)}\cap\fg^{\bar\lam})=\fg_{-\pi(\a)}\cap\fg^{-\bar\lam}$, for $\a\in R$, $\lam\in\Lam$.} 
\end{lem}

\proof {Since $\fg_{\pi(\a)}=\sum_{\{\b\in R\mid\pi(\b)=\pi(\a)\}}\fg_\b$, it follows from Lemma \ref{chev1}(i) that $\tau(\fg_{\pi(\a)})=\fg_{-\pi(\a)}$. Moreover, since $\tau$ commutes with $\sg$, it preserves the grading spaces $\fg^{\bar\lam}$, $\lam\in\Lam$. On the other hand since
$\fg_{\pi(\a)}=\{x\in\fg\mid [h,x]=\a(h)x\hbox{ for }h\in\fh^{\bar 0}\}$, we get $\psi(\fg_{\pi(\a)})\cap\fg^{\bar\lam}=\fg_{\pi(\a)}\cap\fg^{-\bar\lam}$. The result now is immediate from the above facts.}\qed

\begin{lem}\label{chev3}
{Let $\psi$ be an automorphism of $\fg$ and $\{X_{\b}\mid \b\in\bb\sub R\}$ be a set of Lie algebra generators  for the extended affine Lie algebra $\fg$.} If for each  $\b\in\bb$, there exists $\b'\in \bb$ such that $\psi\sg^i(X_{\b})=\sg^{-i}X_{\b'}$ for all $0\leq i\leq m-1$, then $\psi(\fg^{\bar j})\sub\fg^{-\bar j}$, for each
$j$.\end{lem}

\proof
Let $X=[X_{\b_{1}},\ldots, X_{\b_{t}}]$, $\b_i\in\bb$. Then
\begin{eqnarray*}
\psi\pi_j(X)&=&\frac{1}{m}\sum_{i}\omega^{-ij}\psi\sg^i[X_{\b_1},\ldots, X_{\b_t}]\\
&=&\frac{1}{m}\sum_i\omega^{-ij}[\sg^{-i}X_{\b'_1},\ldots,\sg^{-i}X_{\b'_t}]\\
&=&
\frac{1}{m}\sum_i\omega^{ij}\sg^{i}[X_{\b'_1},\ldots,X_{\b'_t}]\\
&=&\frac{1}{m}\pi_{-j}([X_{\b'_1},\ldots,X_{\b'_t}])\in\fg^{-\bar j}.
\end{eqnarray*}
\qed

\begin{pro}\label{pro115}
 {Assume that the extended affine Lie algebra $(\fg,\fm,\fh)$ admits
 a $\sg$-Chevalley pair $(\tau,\psi)$.  
%
%
%
 Then the assignment 
$$\pi_\lam(x)\otimes a^\lam+\gamma+d\mapsto \psi\tau(\pi_{\lam}(x))\otimes a^{-\lam}-\gamma-d,
 $$ 
 ($x\in\fg,$ $\lam\in\Lam,$ $\gamma\in\v,$ $d\in\v^\star$)  
 defines an automorphism ${\bar\tau}_{\psi}$ of $\hat\ll(\fg,\aa)$ which restricts to a Chevalley automorphism of ${\hat\ll}_\rho(\fg,\aa)$.

}
\end{pro}

{\proof {The first statement follows from the fact that for any form preserving automorphism $\Phi$ of $\fg$, the map
 $\bar\Phi$ induced by the assignment 
$\pi_j(x)\otimes a^\lam+\gamma+d\mapsto \Phi(\pi_{j}(x))\otimes a^{-\lam}-\gamma-d,
 $
 defines an automorphism of $\hat\ll(\fg,\aa)$ with inverse 
$\overline{\Phi^{-1}}$. To be more precise, for
$x\in\fg, \lambda\in\Lambda$, we have
$${\bar\Phi}(x\otimes a^{\lambda})=\sum_{j}{\bar\Phi}(\pi_{j}(x)\otimes a^{\lambda})=\sum_{j}\Phi(\pi_{j}(x))\otimes a^{-\lambda})=\Phi(x)\otimes a^{-\lambda}.$$
So for $x,y\in\fg$, $\lam,\lam'\in\Lam$, $\gamma,\gamma'\in\CV$ and $d,d'\in\CV^\star$, we have
\begin{equation*}
	\begin{split}
		{\bar\Phi}[x\otimes a^\lam+\gamma+d,
		y\otimes a^{\lam'}+\gamma'+d']=&{{\bar\Phi}\big([x,y]\otimes a^{\lam+\lam'}}
		+(x,y)\d_{\lam,-\lam'}\lam\big)\\
		&+{\bar\Phi}\big(d(\lam')y\otimes a^{\lam'}-d'(\lam)x\otimes a^{\lam}\big)\\
		=&[\Phi(x),\Phi(y)]\otimes {a^{-\lam-\lam'}}\\
		&+(\Phi(x), \Phi(y))\d_{-\lam,\lam'}(-\lam)\\
		&+d(\lam')\Phi(y)\otimes a^{-\lam'}-d'(\lam)\Phi(x) a^{-\lam}\\
		=&[{\bar\Phi} (x\otimes a^\lam+\gamma+d),{\bar\Phi}(y\otimes a^{\lam'}+\gamma'+d')],
	\end{split}
\end{equation*}
where to get the above equalities we have used the facts that $\Phi$ preserves the form and that $\mu(\lam,\lam')=\mu(-\lam,-\lam')$. 
This shows that $\bar\Phi$ is an automorphism of $\hat\ll(\fg,\CA)$, with inverse $\overline{\Phi^{-1}}$.}}   

{We now know that $\bar\tau_\psi$ is a finite order automorphism of $\hat\ll(\fg,\aa)$. Moreover, since $\tau$ commutes with each $\pi_\lam$, and $\psi$ sends the subspace $\pi_\lam(\fg)$ to $\pi_{-\lam}(\fg)$, we get 
 for $x\in\fg, \lam\in\Lam$, that
$$\bar\tau_{\psi}(\pi_{\lam}(x)\otimes a^{\lam})=\psi(\tau(\pi_{\lam}(x)))\otimes a^{-\lam}=\psi(\pi_{\lam}(\tau(x)))\otimes a^{-\lam}\in\pi_{-\lam}(\fg)\otimes a^{-\lam}.$$
So $\bar\tau_{\psi}(\hat\ll_\rho(\fg,\aa))\subseteq\hat\ll_\rho(\fg,\aa)$. Furthermore, since $\psi_{|_{\pi_0(\fh)}}=\id_{\fh}$, we have for $h\in\fh$,
${\bar\tau}_\psi(\pi_0(h))\otimes 1+d+\lam=\psi\tau(\pi_0(h))\otimes 1-d-\lam=-\pi_0(h)\otimes 1-d-\lam$. Thus
$\bar\tau_{\psi}$ is a Chevalley automorphism of $\hat\ll_\rho(\fg,\aa)$.}\qed

\begin{cor}\label{cor115-0} (Toroidal Lie algebras) 
Let $\fg$ be a finite dimensional simple Lie algebra, and $\tau$ be a Chevalley automorphism of $\fg$. Then 
$\bar\tau:={\bar\tau}_{\id}$ is a Chevalley automorphism of $\hat{L}(\fg,\aa)$. In particular, if $\aa$ is the algebra of Laurent polynomials in $\nu$ variables, then $\bar\tau$ is a Chevalley automorphism 
of the toroidal Lie algebra ${\hat\ll}(\fg,\aa)$.
\end{cor}

\proof {Apply Proposition \ref{pro115} with $\sg=\psi=\id$.}
\qed

\begin{cor}\label{cor115-1}
{Let $(\fg,\fm,\fh)$ be an extended affine Lie algebra and $\sg$ be an automorphism of order $2$. If $\fg$ admits a Chevalley automorphism
$\tau$ commuting with $\sg$, then $\bar\tau:={\bar\tau}_{\id}$ is a Chevalley automorphism of ${\hat\ll}_\rho(\fg,\aa)$.}
\end{cor}

\proof
{Since $\sg$ has order $2$,  $(\tau,\id)$ is a $\sg$-Chevalley pair and so  we are done by Proposition
\ref{pro115}}.\qed

\begin{cor}\label{cor115} (Affine Lie algebras)
{Let $\fg$ be a finite dimensional simple Lie algebra, $\tau$ a Chevalley automorphism of $\fg$ and $\sg$  a non-trivial graph automorphism. Then ${\bar\tau}_\psi$ is a Chevalley automorphism for ${\hat\ll}_\rho(\fg,\aa)$ where $\psi=\id$ if $\sg$ has order $2$, and  $\psi$ is a graph automorphism of order $2$ otherwise.
 In particular if $\aa$
is the algebra of Laurent polynomials in one variable, then ${\bar\tau}_\psi$ is a Chevalley automorphism of the affine Lie algebra ${\hat\ll}_\rho(\fg,\aa)$.}
\end{cor}  

\proof One knows that for $\sg$ and $\tau$ as in the statement $\sg\tau=\tau\sg$.
If $\sg$ has order $2$, then we are done by Corollary \ref{cor115-1}.

Suppose next that $\sg$ has order $3$, namely  $\sg$ is the automorphism of $D_4$ induced by Dynkin diagram automorphism:

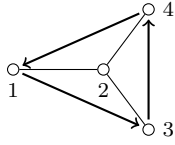
\begin{center}
\scriptsize{
  \begin{tikzpicture}
	\node[dnode,label=below:$2$] (2) at (0,0) {};
	\node[dnode,label=below:$1$] (1) at (-1.2,0) {};
	\node[dnode,label=right:$3$] (3) at (.6,-.8) {};
	\node[dnode,label=right:$4$] (4) at (.6,.8) {};
     
    \path (1) edge[sedge] (2)
          (2) edge[sedge] (3)
              edge[sedge] (4);
    
    \draw[->,shorten >=1.5pt,shorten <=1.5pt,thick] (1) -- (3);
    \draw[->,shorten >=1.5pt,shorten <=1.5pt,thick] (3) -- (4);
    \draw[->,shorten >=1.5pt,shorten <=1.5pt,thick] (4) -- (1);
  \end{tikzpicture}
  }
\end{center}
We take $\psi$ to be be the order $2$ diagram automorphism of $D_4$ induced by the diagram automorphism:
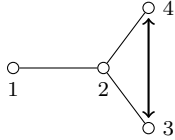
\begin{center}
\scriptsize{
  \begin{tikzpicture}
	\node[dnode,label=below:$2$] (2) at (0,0) {};
	\node[dnode,label=below:$1$] (1) at (-1.2,0) {};
	\node[dnode,label=right:$3$] (3) at (.6,-.8) {};
	\node[dnode,label=right:$4$] (4) at (.6,.8) {};
    
    \path (1) edge[sedge] (2)
          (2) edge[sedge] (3)
              edge[sedge] (4);
    
    \draw[<->,shorten >=1.5pt,shorten <=1.5pt,thick] (3) -- (4);
  \end{tikzpicture}
  }
\end{center}
Now if $X\in\{h_{\a_i}, E_{\pm i}\mid 1\leq i\leq 4\}$ is a standard  Chevalley generator for $\fg$, then we simply check that for $k=0,1,2,$
$$\psi\sg^k(X)=\left\{\begin{array}{ll}
\sg^{-k}X&\hbox{if }X\in\{h_{\a_i},E_{\pm\a_i}\mid i=1,2\},\\
\sg^{-k}\psi(X)&\hbox{if }X\in\{h_{\a_i},E_{\pm\a_i}\mid i=3,4\}.
\end{array}\right.
$$
Thus by Lemma \ref{chev2},  $(\tau,\psi)$ is $\sg$-Chevalley pair (see also \cite[Lemma 3.8.11]{Mit85}). Hence $\bar\tau_{\psi}$ is a Chevalley automorphism of $\hat\ll_\rho(\fg,\aa)$.\qed

\begin{cor}\label{cor115-2} (Elliptic Lie algebras) 
Let $\fg$ be an affine Lie algebra,  $\sg$ be a non-identity graph automorphism and $\tau$ be a Chevalley automorphism of $\fg$. Then
${\bar\tau}_\psi$ is a Chevalley automorphism of ${\hat\ll}_\rho(\fg,\aa)$ for some graph automorphism $\psi$. In particular if $\aa$ is the algebra of Laurent polynomials in one variable then ${\bar\tau}_\psi$ is a Chevalley automorphism of the elliptic Lie algebra ${\hat\ll}_\rho(\fg,\aa)$.
\end{cor}

\proof First, one observes that $\sg\tau=\tau\sg$. There are $22$ non-identity graph automorphisms of affine type, from which $17$ are of order $2$ for them Corollary \ref{cor115-1} applies. Assume next that $m>2$, where $m$ is the order of $\sg$. If $m=3$, then $\sg$ is the automorphism induced by  the order $3$ diagram automorphism of $D_4^{(1)}$ or $E_6^{(1)}$. In what follows, we let $\{h_{\a_i}, E_{\pm\a_i}\}$ be 
the standard Chevalley generators of the affine Lie algebra $\fg$.

{$D_4^{(1)}$: $\sg$ is given by
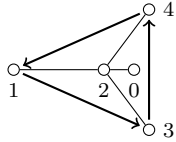
\begin{center}
\begin{scriptsize}
	\begin{tikzpicture}
		\node[dnode,label=below:$2$] (2) at (0,0) {};
		\node[dnode,label=below:$0$] (0) at (.4,0) {};
		\node[dnode,label=below:$1$] (1) at (-1.2,0) {};
		\node[dnode,label=right:$3$] (3) at (.6,-.8) {};
		\node[dnode,label=right:$4$] (4) at (.6,.8) {};
		
		\path (1) edge[sedge] (2)
		   	  (2) edge[sedge] (3)
				  edge[sedge] (4)
				  edge[sedge] (0);
		
		\draw[->,shorten >=1.5pt,shorten <=1.5pt,thick] (1) -- (3);
		\draw[->,shorten >=1.5pt,shorten <=1.5pt,thick] (3) -- (4);
		\draw[->,shorten >=1.5pt,shorten <=1.5pt,thick] (4) -- (1);
	\end{tikzpicture}
\end{scriptsize}

\end{center}
{We  let $\psi$ be the automorphism induced by order $2$ diagram automorphism:
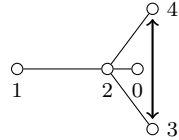
\begin{center}
\scriptsize{  
	\begin{tikzpicture}
	\node[dnode,label=below:$2$] (2) at (0,0) {};
	\node[dnode,label=below:$0$] (0) at (.4,0) {};
	\node[dnode,label=below:$1$] (1) at (-1.2,0) {};
	\node[dnode,label=right:$3$] (3) at (.6,-.8) {};
	\node[dnode,label=right:$4$] (4) at (.6,.8) {};
    	
    	\path (1) edge[sedge] (2)
          	  (2) edge[sedge] (3)
              	  edge[sedge] (4)
               	  edge[sedge] (0);
    
    	\draw[<->,shorten >=1.5pt,shorten <=1.5pt,thick] (3) -- (4);
  \end{tikzpicture}
  }
\end{center}
Then we simply check that for $k=0,1,2,$
$$\psi\sg^k(X)=\left\{\begin{array}{ll}
\sg^{-k}X&\hbox{if }X\in\{h_{\a_i},E_{\pm\a_i}\mid i=0,1,2\},\\
\sg^{-k}\psi(X)&\hbox{if }X\in\{h_{\a_i},E_{\pm\a_i}\mid i=3,4\}.
\end{array}\right.
$$
Thus by Lemma \ref{chev2},  $(\tau,\psi)$ is $\sg$-Chevalley pair. So $\bar\tau_{\psi}$ is a Chevalley automorphism of $\hat\ll_\rho(\fg,\aa)$.

$E_6^{(1)}$: $\sg$ is given by\begin{center}
\scriptsize
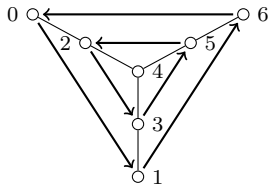
{  \begin{tikzpicture}
    \node[dnode,label=right:$4$] (4) at (0,0) {};
    \node[dnode,label=left:$2$] (2) at (-.7,.38) {};
    \node[dnode,label=right:$3$] (3) at (0,-.7) {};
    \node[dnode,label=right:$5$] (5) at (.7,.38) {};
    \node[dnode,label=right:$1$] (1) at (0,-1.4) {};
    \node[dnode,label=left:$0$] (0) at (-1.4,.76) {};
    \node[dnode,label=right:$6$] (6) at (1.4,.76) {};
    
    \path (4) edge[sedge] (2)
              edge[sedge] (3)
              edge[sedge] (5)
          (3) edge[sedge] (1)
          (2) edge[sedge] (0)
          (5) edge[sedge] (6);
          
    \draw[->,shorten >=1.5pt,shorten <=1.5pt,thick] (2) -- (3);
    \draw[->,shorten >=1.5pt,shorten <=1.5pt,thick] (3) -- (5);
    \draw[->,shorten >=1.5pt,shorten <=1.5pt,thick] (5) -- (2);
    \draw[->,shorten >=1.5pt,shorten <=1.5pt,thick] (0) -- (1);
    \draw[->,shorten >=1.5pt,shorten <=1.5pt,thick] (1) -- (6);
    \draw[->,shorten >=1.5pt,shorten <=1.5pt,thick] (6) -- (0);
  \end{tikzpicture}
  }
\end{center}
We take $\psi$ to be the automorphism induced by order $2$ diagram automorphism:
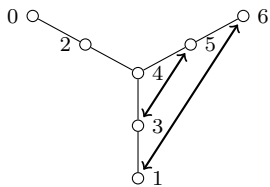
\begin{center}
 \scriptsize{ \begin{tikzpicture}
    \node[dnode,label=right:$4$] (4) at (0,0) {};
	\node[dnode,label=left:$2$] (2) at (-.7,.38) {};
	\node[dnode,label=right:$3$] (3) at (0,-.7) {};
	\node[dnode,label=right:$5$] (5) at (.7,.38) {};
	\node[dnode,label=right:$1$] (1) at (0,-1.4) {};
	\node[dnode,label=left:$0$] (0) at (-1.4,.76) {};
	\node[dnode,label=right:$6$] (6) at (1.4,.76) {};
    
    \path (4) edge[sedge] (2)
              edge[sedge] (3)
              edge[sedge] (5)
          (3) edge[sedge] (1)
          (2) edge[sedge] (0)
          (5) edge[sedge] (6);
          
    \draw[<->,shorten >=1.5pt,shorten <=1.5pt,thick] (3) -- (5);
    \draw[<->,shorten >=1.5pt,shorten <=1.5pt,thick] (1) -- (6);
  \end{tikzpicture}
  }
\end{center}
Then one directly checks that for $k=0,1,2,$
$$
\psi\sg^k(X)=\left\{\begin{array}{ll}
\sg^{-k}X&\hbox{if }X\in\{h_{\a_i},E_{\pm\a_i}\mid i=0,2,4\},\\ 
\sg^{-k}\psi(X)&\hbox{if }X\in\{h_{\a_i},E_{\pm\a_i}\mid i=1,3,5,6\}.
\end{array}\right.
$$
Thus by Lemma \ref{chev2},  $(\tau,\psi)$ is $\sg$-Chevalley pair. Hence $\bar\tau_{\psi}$ is a Chevalley automorphism of $\hat\ll_\rho(\fg,\aa)$.

Assume next that $m=4$. Then $\sg$ is the automorphism induced by one of diagram automorphisms  $D_4^{(1)}$, $D_{2\ell+2}^{(1)}$, or $D_{2\ell+1}^{(1)}$, $\ell>1$. In what follows for each case we introduce the appropriate $\psi$ and show that it satisfies conditions of Lemma \ref{chev3}, with respect to the standard Chevalley generators $\{h_{\a_i}, E_{\pm\a_i}\}$ of the affine Lie algebra $\fg$:

$D_4^{(1)}$:

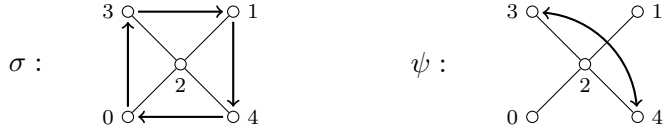
\begin{center}
\begin{tikzpicture}
    \draw (-1.7,0) node[anchor=east]  {$\sigma:$};
    
    \scriptsize{
    \node[dnode,label=below:$2$] (2) at (0,0) {};
    \node[dnode,label=left:$0$] (0) at (-.7,-.7) {};
    \node[dnode,label=right:$1$] (1) at (.7,.7) {};
    \node[dnode,label=left:$3$] (3) at (-.7,.7) {};
    \node[dnode,label=right:$4$] (4) at (.7,-.7) {};
    
    \path (2) edge[sedge] (0)
              edge[sedge] (1)
              edge[sedge] (3)
              edge[sedge] (4);
              
    \draw[->,shorten >=1.5pt,shorten <=1.5pt,thick] (1) -- (4);
    \draw[->,shorten >=1.5pt,shorten <=1.5pt,thick] (4) -- (0);
    \draw[->,shorten >=1.5pt,shorten <=1.5pt,thick] (0) -- (3);
    \draw[->,shorten >=1.5pt,shorten <=1.5pt,thick] (3) -- (1);}
\end{tikzpicture}
$\quadd\quadd$
\begin{tikzpicture}
	\draw (-1.7,0) node[anchor=east]  {$\psi:$};
	
	\scriptsize{
	\node[dnode,label=below:$2$] (2) at (0,0) {};
	\node[dnode,label=left:$0$] (0) at (-.7,-.7) {};
	\node[dnode,label=right:$1$] (1) at (.7,.7) {};
	\node[dnode,label=left:$3$] (3) at (-.7,.7) {};
	\node[dnode,label=right:$4$] (4) at (.7,-.7) {};
    
    \path (2) edge[sedge] (0)
              edge[sedge] (1)
              edge[sedge] (3)
              edge[sedge] (4);
              
    \draw[<->,shorten >=1.5pt,shorten <=1.5pt,thick] (3) edge[bend left=40] (4);}
\end{tikzpicture}
\end{center}

$$\psi\sg^k(X)=\left\{\begin{array}{ll}
\sg^{-k}X&\hbox{if }X\in\{h_{\a_i},E_{\pm\a_i}\mid i=0,1,2\},\\
\sg^{-k}\psi(X)&\hbox{if }X\in\{h_{\a_i},E_{\pm\a_i}\mid i=3,4\},
\end{array}\right.
\quad(k=0,1,2,3).
$$

$D_{2l+2}^{(1)}$:
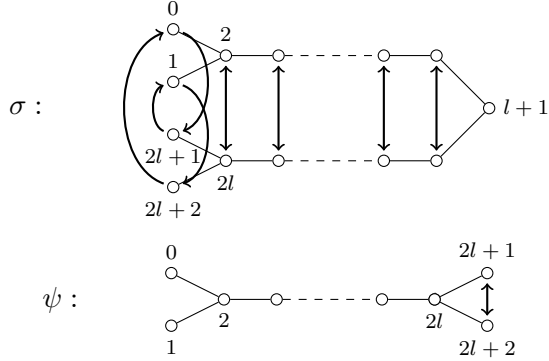
\begin{center}  
\begin{tikzpicture}
	\draw (-3.7,0) node[anchor=east]  {$\sigma:$};
	
	\scriptsize{
    \node[dnode] (A) at (-.7,-.7) {};
    \node[dnode] (B) at (.7,.7) {};
    \node[dnode] (C) at (-.7,.7) {};
    \node[dnode] (D) at (.7,-.7) {};
    \node[dnode] (E) at (1.4,.7) {};
    \node[dnode] (F) at (1.4,-.7) {};
    \node[dnode,label=right:$l+1$] (l+1) at (2.1,0) {};
    \node[dnode,label=above:$2$] (2) at (-1.4,.7) {};
    \node[dnode,label=below:$2l$] (2l) at (-1.4,-.7) {};
    \node[dnode,label=above:$0$] (0) at (-2.1,1.05) {};
    \node[dnode,label=above:$1$] (1) at (-2.1,.35) {};
    \node[dnode,label=below:$2l+1$] (2l+1) at (-2.1,-.35) {};
    \node[dnode,label=below:$2l+2$] (2l+2) at (-2.1,-1.05) {};
    
    \path (2) edge[sedge] (0)
              edge[sedge] (1)
              edge[sedge] (C)
         (2l) edge[sedge] (A)
              edge[sedge] (2l+1)
              edge[sedge] (2l+2)
          (B) edge[sedge] (E)
          (D) edge[sedge] (F)
        (l+1) edge[sedge] (E)
              edge[sedge] (F);
              
    \draw[<->,shorten >=1.5pt,shorten <=1.5pt,thick] (2) -- (2l);
    \draw[<->,shorten >=1.5pt,shorten <=1.5pt,thick] (A) -- (C);
    \draw[<->,shorten >=1.5pt,shorten <=1.5pt,thick] (B) -- (D);
    \draw[<->,shorten >=1.5pt,shorten <=1.5pt,thick] (E) -- (F);
    \draw[dashed] (C) -- (B);
    \draw[dashed] (A) -- (D);
    \draw[->,shorten >=1.5pt,shorten <=1.5pt,thick] (2l+2) edge[bend left=70] (0);
    \draw[->,shorten >=1.5pt,shorten <=1.5pt,thick] (2l+1) edge[bend left=70] (1);
    \draw[->,shorten >=1.5pt,shorten <=1.5pt,thick] (0) edge[bend left=70] (2l+1);
    \draw[->,shorten >=1.5pt,shorten <=1.5pt,thick] (1) edge[bend left=70] (2l+2);}
\end{tikzpicture}
\end{center}
\begin{center}
\begin{tikzpicture}
	\draw (-3.2,0) node[anchor=east]  {$\psi:$};
	
    \scriptsize{ 
    \node[dnode] (B) at (.7,0) {};
    \node[dnode] (C) at (-.7,0) {};
    \node[dnode] (E) at (1.4,0) {};
    \node[dnode,label=below:$2$] (2) at (-1.4,0) {};
    \node[dnode,label=below:$2l$] (2l) at (1.4,0) {};
    \node[dnode,label=above:$0$] (0) at (-2.1,.35) {};
    \node[dnode,label=below:$1$] (1) at (-2.1,-.35) {};
    \node[dnode,label=above:$2l+1$] (2l+1) at (2.1,.35) {};
    \node[dnode,label=below:$2l+2$] (2l+2) at (2.1,-.35) {};
    
    \path (2) edge[sedge] (0)
              edge[sedge] (1)
              edge[sedge] (C)
         (2l) edge[sedge] (E)
              edge[sedge] (2l+1)
              edge[sedge] (2l+2)
          (B) edge[sedge] (E);
              
    \draw[dashed] (C) -- (B);
    \draw[<->,shorten >=1.5pt,shorten <=1.5pt,thick] (2l+2) -- (2l+1);}
\end{tikzpicture}
\end{center}
$$\psi\sg^k(X)=\sg^{-k}\psi(X),\quad(k=0,1,2,3).$$

$D_{2l+3}^{(1)}$:
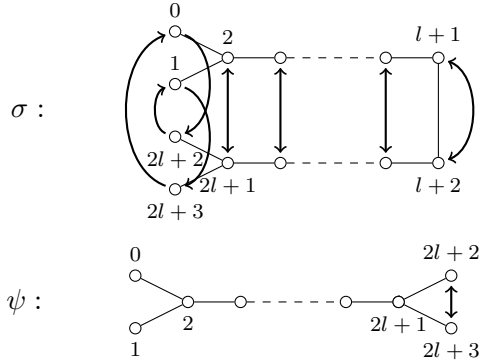
\begin{center}
\begin{tikzpicture}
    \draw (-3.7,0) node[anchor=east]  {$\sigma:$};
    
    \scriptsize{  
    \node[dnode] (A) at (-.7,-.7) {};
    \node[dnode] (B) at (.7,.7) {};
    \node[dnode] (C) at (-.7,.7) {};
    \node[dnode] (D) at (.7,-.7) {};
    \node[dnode,label=above:$l+1$] (l+1) at (1.4,.7) {};
    \node[dnode,label=below:$l+2$] (l+2) at (1.4,-.7) {};
    \node[dnode,label=above:$2$] (2) at (-1.4,.7) {};
    \node[dnode,label=below:$2l+1$] (2l+1) at (-1.4,-.7) {};
    \node[dnode,label=above:$0$] (0) at (-2.1,1.05) {};
    \node[dnode,label=above:$1$] (1) at (-2.1,.35) {};
    \node[dnode,label=below:$2l+2$] (2l+2) at (-2.1,-.35) {};
    \node[dnode,label=below:$2l+3$] (2l+3) at (-2.1,-1.05) {};
    
    \path (2) edge[sedge] (0)
              edge[sedge] (1)
              edge[sedge] (C)
       (2l+1) edge[sedge] (A)
              edge[sedge] (2l+2)
              edge[sedge] (2l+3)
          (B) edge[sedge] (l+1)
          (D) edge[sedge] (l+2)
          (l+1) edge[sedge] (l+2);
              
    \draw[<->,shorten >=1.5pt,shorten <=1.5pt,thick] (2) -- (2l+1);
    \draw[<->,shorten >=1.5pt,shorten <=1.5pt,thick] (A) -- (C);
    \draw[<->,shorten >=1.5pt,shorten <=1.5pt,thick] (B) -- (D);
    \draw[<->,shorten >=1.5pt,shorten <=1.5pt,thick] (l+1) edge[bend left=70] (l+2);
    \draw[dashed] (C) -- (B);
    \draw[dashed] (A) -- (D);
    \draw[->,shorten >=1.5pt,shorten <=1.5pt,thick] (2l+3) edge[bend left=70] (0);
    \draw[->,shorten >=1.5pt,shorten <=1.5pt,thick] (2l+2) edge[bend left=70] (1);
    \draw[->,shorten >=1.5pt,shorten <=1.5pt,thick] (0) edge[bend left=70] (2l+2);
    \draw[->,shorten >=1.5pt,shorten <=1.5pt,thick] (1) edge[bend left=70] (2l+3);}
\end{tikzpicture}
\end{center}
\begin{center}
\begin{tikzpicture}
	\draw (-3.2,0) node[anchor=east]  {$\psi:$};
	
	\scriptsize{ 
    \node[dnode] (B) at (.7,0) {};
    \node[dnode] (C) at (-.7,0) {};
    \node[dnode] (E) at (1.4,0) {};
    \node[dnode,label=below:$2$] (2) at (-1.4,0) {};
    \node[dnode,label=below:$2l+1$] (2l+1) at (1.4,0) {};
    \node[dnode,label=above:$0$] (0) at (-2.1,.35) {};
    \node[dnode,label=below:$1$] (1) at (-2.1,-.35) {};
    \node[dnode,label=above:$2l+2$] (2l+2) at (2.1,.35) {};
    \node[dnode,label=below:$2l+3$] (2l+3) at (2.1,-.35) {};
    
    \path (2) edge[sedge] (0)
              edge[sedge] (1)
              edge[sedge] (C)
         (2l+1) edge[sedge] (E)
              edge[sedge] (2l+2)
              edge[sedge] (2l+3)
          (B) edge[sedge] (E);
              
    \draw[dashed] (C) -- (B);
    \draw[<->,shorten >=1.5pt,shorten <=1.5pt,thick] (2l+3) -- (2l+2);}
\end{tikzpicture}
\end{center}
$$\psi\sg^k(X)=\sg^{-k}\psi(X),\quad (k=0,1,2,3).$$
Considering the above facts, $(\tau,\psi)$ is $\sg$-Chevalley pair (Lemma \ref{chev2}) and so $\bar\tau_{\psi}$ is a Chevalley automorphism of $\hat\ll_\rho(\fg,\aa)$.\qed

\begin{exa}\label{ex3}
	Let $\nu\geq 1$ and 
$\aa=\sum_{\underline{n}\in\bbbz^\nu}\bbbc x_1^{n_1}\cdots x_{\nu}^{n_\nu},$ $\underline{n}=(n_1\ldots,n_\nu)$,
be the algebra of Laurent polynomials in $\nu$ variables.
	Suppose $\Lam=\bbbz^\nu$, $m\geq 1$ and $\tau_{1},\ldots,\tau_{m}$
	represent distinct cosets of $2\Lam$ in $\Lam$ with $\tau_1=0$.
	
	Next let $l\geq 1$ and put
	$$
	F=\left[\begin{array}{lll}
	x^{\tau_{1}}&\cdots&0\\
	\vdots&\ddots&\vdots\\
	0&\cdots&x^{\tau_{m}}\end{array}\right]\andd
	K=\left[\begin{array}{ccc}
	0&I_{\ell}&0\\
	I_{\ell}&0&0\\
	0&0&F\end{array}\right].
	$$
	{Then $F$ is an  invertible $m\times m$-matrix and $K$ is an
	invertible $n\times n$-matrix, where $n=2\ell+m$. Set
	$\dot\fg=\sl{sl}_n (\bbbc)\otimes\aa$ which we identify it with 
	${\sl{sl}_n} (\aa):=\{X\in M_{n}(\aa)\mid\hbox{tr}(X)=0\}$, the subalgebra of the Lie algebra $\sl{gl}_n(\aa)$ with underlying space $M_n(\aa)$ and the commutator product. 
	We define the ${\bbbc}-$linear
map $\ep:\aa\longrightarrow{\bbbc}$ induced by
$\ep(x^{\sg})=1$ if $\sg=0$ and $\ep(x^{\sg})=0$ if $\sg\not=0$.
This provides a non-degenerate associative symmetric bilinear form $\fm$ on $\sl{gl}_{n}(\aa)$ defined by
\begin{equation}\label{r2}
(A,B)=\ep(tr(AB)).
\end{equation}
 Put
$\dot\fh=\span_{\bbbc}\{e_{ii}-e_{i+1,i+1}\mid 1\leq i\leq\ell\}$. Next we set
$$\fg=\dot\fg\oplus\v\oplus\v^\star,\andd\fh=\dot\fh\oplus\v\oplus
\v^\star,$$
where $\v:=\Lam\otimes\bbbc$. Then exactly as in (\ref{bracketloop}) and (\ref{formloop}), we extend the bracket and the form on $\dot\fg$ to $\fg$. It then follows that
$(\fg,\fm,\fh)$ is an extend affine Lie algebra of type $A_{\ell-1}$ (see \cite[Example III.1.29]{AABGP97}).}

{We now consider endomorphisms $\sg$ and $\tau$ of $\fg$ given by
$$\begin{array}{l}
\sg(X)=-K^{-1}{X}^{t}K,\quad\sg_{|_{\v\oplus\v^\star}}=\id\\
\tau(x^\lam E_{p,q})=-x^{-\lam}E_{q,p},\quad
\tau_{|_{\v\oplus\v^\star}}=-id,
\end{array}
$$
where $\{E_{p,q}\mid 1\leq p,q\leq n\}$ denotes the set of matrix units, and $\lam\in\Lam$.
It is straightforward to check that both $\sg$ and $\tau$ are in fact Lie algebra automorphisms of order $2$, with $\tau_{_\fh}=-\id$. Moreover, one checks
directly that they commute. 
 Thus by Corollary \ref{cor115-1}, $\bar\tau$ is a Chevalley automorphism of $\hat\fg$.}
We note that the type $X$ of $\hat\fg$ is given by
$$X=\left\{\begin{array}{ll}
A_{1}&\hbox{if }l=1,\\
B_{\ell}&\hbox{if }l\geq 2.
\end{array}\right.
$$
\end{exa}

We conclude this section with a result concerning  diagonal automorphisms of an extended affine Lie algebra, namely an automorphism
of the form $\sg(x_\a)=\Phi(\a)x_\a$ for $x_\a\in\ll_\a$, $\a\in R$, where $\Phi:\la R\ra\rightarrow\bbbc^\times$ is a group homomorphism.

\begin{lem}\label{lemdia}
Assume that $\sg$ is a non-identity finite order {\it diagonal} automorphism of an extended affine  Lie algebra $\fg$ and $\tau$ is a Chevalley automorphism.Then
 
(i) $\sg\tau_{|_\fh}=\tau\sg_{|_\fh}=-\id_{_\fh}$, in particular
 $\sg\tau$ is a Chevalley automorphism.
 
(ii) $\sg\tau=\tau\sg$ if and only if $\sg$ has order $2$; in which case
 $(\tau,\id)$ is a Chevalley pair.
 \end{lem}
 
 \proof  Part (i) is clear as $\sg_{|_\fh}=\id_\fh$. 
 
(ii) Assume that $\sg$ is given by the group homomorphism
 $\Psi:\la R\ra\rightarrow \bbbc^\times$. 
If $\sg$ has order $2$, then
$\Psi(\a)=\Psi(-\a)$ for each $\alpha$ and so $\sg\tau=\tau\sg$.
If $\sg$ has order $>2$, then there exists a root $\a$ such that
$\Psi(\a)\not=\Psi(-\a).$ Then for $0\not=x_\a\in\fg_\a$, $\sg\tau(x_\a)=\Psi(-\a)\tau(x_\a)\not=\Psi(\a)\tau(x_\a)=\tau\sg(x_\a).$\qed

\section{Preliminary results on integral structure}\label{Preliminary results on integral structure}\setcounter{equation}{0}
We  use our knowledge on $\bbbz$-forms of  finite and affine Lie algebras to provide some preliminary  results needed for Section \ref{Integral structure of the core}, in which we equip with an integral structure the core of an extended affine Lie algebra admitting a Chevalley automorphism.

Throughout this section, we assume that $(\fg,\fm,\fh)$ is a tame reduced 
extended affine Lie algebra with root system $R$. 
{We also assume that $\fg$ is equipped with a Chevalley automorphism
$\tau$.}
We recall from (\ref{eq99-1}) that 
\begin{equation}\label{sss}
R=(S+S)\cap(\rd_{sh}+S)\cup (\rd_{lg}+L),
\end{equation}
where $S$ and $L$ are two semilattices and $\rd$ is a finite root system.
We fix a base $\dot\Pi$ of $\rd$. Also we consider two subsets
$R^{\pm}$ of $R^\times$, such that
\begin{equation}\label{eq2}
 R^\times=R^+\uplus R^-\quad\hbox{and}\quad
R^+=-R^-.
\end{equation}
We call such a decomposition, a  {\it positivity} decomposition
for $R^\times$.

\begin{rem}\label{remeq2}
The existence of a positivity decomposition can be guaranteed in several ways. We present two of them here. Let $\rd^\times=\rd^+\uplus\rd^-$ be the decomposition of $\rd$ into positive and negative roots with respect to a base of $\rd$, and set
$$R^+ := (\rd^+ +\Lam)\cap R^\times\andd
R^- :=(\rd^- +\Lam)\cap R^\times.
$$
Then the sets $R^\pm$ provide a positivity decomposition for $R^\times.$

The second way which we present here is motivated from
I. G. Macdonald's approach  (see \cite{Mac72})
 in decomposing an affine root system into positive and negative roots. We again fix a decomposition
$\rd^\times=\rd^+\uplus\rd^-$ as above.
For $\a\in\rd^\times$, let 
 $$R^\times_\a:= ((\pm\a+\Lam)\cap R^\times)\andd
 R_\a:=R^\times\cup((R_\a^\times- R_\a^\times)\cap R^0).$$
 By \cite[Lemma 2.1]{APari21},  $R_\a$ is an extended affine root system of type $A_1$.
 In \cite{AK19},  similar to I. G. Macdonald, the authors obtain a decomposition $R_\a^\times=R_\a^+\cup R_\a^-$,
 where $R^+_\a=-R^-_\a$. Note that if $\a,\b\in \rd^+$ with
 $\a\not=\b$, then $R_\a^\times\cap R^\times_\b=\emptyset$.
 We set 
 $$R^+:=\cup_{\a\in\rd^+}R^+_\a\andd R^-:=\cup_{\a\in\rd^-}R^-_\a.$$
 Then as $R^\times={\cup_{\a\in\rd^+}R^\times_{\a}}$, we get
 \begin{equation}\label{temp123}
 R^\times=R^+\uplus R^-\quad\hbox{with}\quad R^+=-R^-.
 \end{equation}
 \end{rem}

We next recall the notion of a nilpotent pair for $R$ introduced in \{\cite[\S 3]{AP19}}.
\begin{DEF}\label{nil}
A pair $\{\a,\b\}$ of roots in $R^\times$ is called 
a {\it nilpotent} pair if $\a+\b\in R^\times$.
\end{DEF}

\begin{lem}\label{hump1}\label{lemnew1} 
Let $\{\a,\b\}$ be a nilpotent pair in $R$. 
Let $u$ and $d$ be the up and down non-negative integers  
appearing in the $\a$-string through $\b$. Then
$d+1 = u(\a+\b,\a+\b)/(\b,\b)$.
\end{lem}

\proof By \cite[Lemma 3.5]{AP19}, $R_{\a,\b}:=R\cap(\bbbz\a\oplus\bbbz\b)$ is an irreducible reduced finite subsystem of $R$. Now the result follows from \cite[Proposition 25.1]{Hum80}.\qed


\begin{pro}\label{hump3}
 Let $(\fg,\fm,\fh)$ be an extended affine Lie algebra of reduced type,
 admitting a Chevalley automorphism $\tau$.
Then for each $\a\in R^\times$, one can choose {root vectors 
$x_{\pm\a}\in\fg_{\pm\a}$} satisfying the following:

(i)  $[x_\a,x_{-\a}]=h_\a,$
where $h_\a:=2t_\a/(\a,\a)$ (see Lemma \ref{oldlem}).

(ii) If $\{\a,\b\}$ is a nilpotent pair in $R$, then 
$[x_\a,x_\b]=c_{\a,\b}x_{\a+\b}$, with
$c_{\a,\b}=-c_{-\a,-\b}$ and
$c_{\a,\b}^2=u(d+1)(\a+\b,\a+\b)/(\b,\b)$,
where $u$ and $d$ are the up and down non-negative integers  
appearing in the $\a$-string through $\b$.  In particular,
$c_{\a,\b}=\pm(d+1)\in\bbbz$, and $\{c_{\a,\b}\mid\{\a,\b\}\hbox{ a nilpotent pair in }R\}$ is a finite set.

\end{pro}

\proof We consider a positivity decomposition $R^\times=R^+\uplus R^-$ as in (\ref{eq2}). For $\a\in R^+$, we pick
a non-zero $x'_\a\in\fg_\a$ and set $x'_{-\a}=-\tau(x'_\a)$.
By (\ref{oldlem}), $[x'_\a,x'_{-\a}]=(x'_\a,x'_{-\a})t_\a$ where $(x'_\a, x'_{-\a})\not=0$. Replacing $x'_\a$ with $x_\a:=cx'_\a$ where
the constant $c$ satisfies $c^2=2/(\a,\a)(x'_\a,x'_{-\a})$, and setting
$x_{-\a}=-\tau(x_\a)$ we get $[x_\a,x_{-\a}]=2t_\a/(\a,\a)$. 

Now let $\{\a,\b\}$ be a nilpotent pair in $R$. Then
$[x_\a,x_\b]=c_{\a,\b}x_{\a+\b}$ for some non-zero $c_{\a,\b}\in\bbbc$.
Applying the Chevalley automorphism $\tau$ to both sides of the latter equality  gives
$c_{\a,\b}=-c_{-\a,-\b}$. 

Next, we set
\begin{equation}\label{simple1}
R_{\a,\b}=(\bbbz\a+\bbbz\b)\cap R\andd
\fg_{\a,\b}=\sum_{\gamma\in R_{\a,\b}^\times}\fg_\gamma\oplus
\sum_{\gamma\in R_{\a,\b}}[\fg_\gamma,\fg_{-\gamma}].
\end{equation}
Then $R_{\a,\b}$ is an irreducible finite root system and $\fg_{\a,\b}$ is a finite dimensional simple Lie algebra of the same type of $R_{\a,\b}$ (see {\cite[Lemma 3.5]{AP19}}).
{Since the Cartan subalgebra $\sum_{\gamma\in R_{\a,\b}}[\fg_\gamma,\fg_{-\gamma}]$ of $\fg_{\a,\b}$ is contained in $\fh$, the automorphism $\tau$ restricts to a Chevalley automorphism of $\fg_{\a,\b}$.} Now applying {\cite[Proposition 25.2]{Hum80}} to
the Lie algebra $\fg_{\a,\b}$, we get that
$c_{\a,\b}^2=u(d+1)(\a+\b,\a+\b)/(\b,\b)$.
This together with Lemma \ref{lemnew1}, gives
$c_{\a,\b}=\pm(d+1)$ and that the set $\{c_{\a,\b}\mid \{\a,\b\}\hbox{ a nilpotent pair in }R\}$ is finite.\qed

\begin{DEF}\label{chevalley-sys}
{Let $(\fg,\fm,\fh)$ be an extended affine Lie algebra and $\tau$ be a Chevalley automorphism for $\fg$. A set $\mathcal{C}=\{x_{\a}\in\fg_\a\mid\a\in R^\times\}$ is called a {\it Chevalley system} for $(\fg,\tau)$ or simply for $\fg$ if}

	- $[x_\a,x_{-\a}]=h_\a:=2t_\a/(\a,\a)$,

	- $\tau(x_\a)=-x_{-\a}$,

\noindent {for all $\a\in R^\times$.}
\end{DEF}                         

{We note that by Proposition {\ref{hump3}(i),} it is possible to choose a Chevalley system $\mathcal{C}=\{x_{\a}\in\fg_\a\mid\a\in R^\times\}$ for $(\fg,\tau)$. If $\mathcal{C}^\prime=\{y_{\a}\in\fg_\a\mid\a\in R^\times\}$ is another Chevalley system for $\tau$, then clearly we have $x_\a=\pm y_\a$, for each $\a\in R^\times$. Also note that any Chevalley system for $\tau$ satisfies condition (ii) of Proposition \ref{hump3}. From now on we fix a Chevalley system $\mathcal{C}=\{x_{\a}\in\fg_\a\mid\a\in R^\times\}$ for $(\fg,\tau)$.}

\begin{lem}\label{newlem34}
Let $\d\in R^0$. Fix a finite root system $\rd$ as in (\ref{eq99}) such that $\d+\dot\a\in R$ for some $\dot\a\in\rd^\times.$ Set
\begin{equation}\label{affine1}
R^\times_{\dot R,\d}:=(\rd+\bbbz\d)\cap R^\times,\quad R_{\dot R,\d}:=\big((\rd+\bbbz\d)\cap R^\times\big)\cup \big((R^\times_{\rd,\d}-R^\times_{\rd,\d})\cap R^0\big),$$
and
$$\fg^a_{\rd,\d}:=\la\fg_\a\mid\a\in R^\times_{\rd,\d}\ra\oplus\bbbc d,
\end{equation}
where $d$ is an appropriately chosen element of $\fh$.
Then $R_{\rd,\d}$ is an affine root subsystem of $R$ and $\fg^a_{\rd,\a}$ is an affine Lie subalgebra of $\fg$ with root system $R_{\rd,\d}$. Moreover,
$\dim(\fg_\d)\geq {\dim(\fg^a_{\rd,\d})_\d}$.
\end{lem}

\proof
The first claim of the statement is just {\cite[Lemma 3.4]{AP19-1}}).\ The second claim is then immediate as 
$$(\fg^a_{\rd,\d})_\d=\sum_{\a\in R^\times_{\rd,\d}}[\fg_{\a+\d},\fg_{-\a}]\sub
\sum_{\a\in R^\times}[\fg_{\a+\d},\fg_{-\a}]\sub\fg_\d.$$
\qed

\begin{rem}\label{rem34} (i) The choice of finite root system $\rd$ in Lemma \ref{newlem34} satisfying $\dot\a+\d\in R$ for some $\dot\a\in\rd$ is possible by Lemma \ref{isos}.

(ii) In {\cite[Proposition 1.4.2]{ABFP09}}, an upper bound is obtained for the dimension of isotropic root spaces of centerless core of an extended affine Lie algebra. Lemma \ref{newlem34} gives a lower bound. In Theorem \ref{thmdim} below, we obtain an upper bound for the isotropic soot spaces of the core.
\end{rem}

{For the sake of simplicity, in the sequel we use the notation
$$x_{\d}^\a:=[x_{\a+\d},x_{-\a}],\qquad (\a\in R^\times,\; 0\not=\d\in R^0,\;\a+\d\in R).
$$
(By convention we set $x_{\d}^\a=0$ if $\a+\d\not\in R$.)}

 \begin{lem}\label{tec1}{ (i) If $\a,\b\in R^\times$, $0\not=\d\in R^0$,
with $\a+\b\in R^\times$ and $\a+\b+\d\in R$, then
 $x^{\a+\b}_\d\in\span\{x^{\a}_\d,x^{\b}_\d\}$ and
 $\bbbc x^{\a}_\d=\bbbc x^{-\a}_\d$.}
 
 {(ii) 
 $\fg_\d\cap\fg_c=\span\{x^{\dot\a+\eta}_\d\mid\dot\a\in\dot\Pi,\;\eta\in S\}.$}
 \end{lem}
 
 \proof (i) Let $\a,\b$ and $\d$ be as in  the statement. We have
 \begin{eqnarray*}
 x^{\a+\b}_\d&=&[x_{\a+\b+\d}, x_{-\a-\b}]\\
 &\in&\bbbc [x_{\a+\b+\d},[x_{-\a},x_{-\b}]]\\
 &\in&\bbbc \big(-[x_{-\b},[x_{\a+\b+\d},x_{-\a}]]-
 {[x_{-\a},[x_{-\b},x_{\a+\b+\d}]]}\big)\\
&\in&\bbbc x^\b_\d+\bbbc x^\a_\d.
\end{eqnarray*}

Assume now that  $\a+\d\in R $. Then there exists $x_\d\in\fg_\d$ such that $x_{\a+\d}=[x_\a,x_\d]$. Then
 \begin{eqnarray*}
 x^{\a}_\d&=&[x_{\a+\d},x_{-\a}]\\
&=& [[x_\a.x_\d],x_{-\a}]=-[[x_{-\a},x_\a],x_\d]-
[[x_\d,x_{-\a}],x_\a]\\
&=&0+[[x_\d,x_{-\a}],x_\a]\in\bbbc x^{-\a}_\d.
\end{eqnarray*}

(ii) Assume first that $R$ is of simply laced type. If $\a\in R^\times$, then
$\a=\dot\a+\eta$ for some $\dot\a\in \rd$ and some $\eta\in R^0$.
Then $x^\a_\d=x^{\dot\a+\eta}_\d.$ We now write
$\dot\a=\dot\a_1+\dot\b$, where $\dot\a_1\in\dot\Pi$ and $\dot\b\in\rd.$ Since $R$ is of simply laced type, $\dot\a_1+\eta\in R$. Then part (i) shows that
$x^\a_\d\in\bbbc x^{\dot\b}_\d+{\bbbc x^{\dot\a_1+\eta}_\d}$.  
Repeating the same argument with $\dot\b$ in place of $\dot\a$, we get the result inductively.

Next, assume  that $R$ is of non-simply laced types. Let $\a=\dot\a+\eta\in R^\times$, $\dot\a\in \rd$, $\eta\in R^0$. Using the realizations of finite root systems, 
one observes that each root in $\rd$ is either short or sum of two  short roots, therefore $\dot\a=\dot\b+\dot\gamma$, where $\dot\b,\dot\gamma\in\rds$. Moreover, from (\ref{sss}), we know $\dot\b+\eta\in R$. Then using  part (i), we get
$x^\a_\d\in\bbbc x^{\dot\gamma}_\d+\bbbc x^{\dot\b+\eta}_\d$. This shows $\fg_\d\cap\fg_c$ is spanned by elements $x^{\dot\a+\eta}_\d$, $\dot\a\in\rds$, $\eta\in S$. Now one can see that each non-simple short root can be written as a sum of simple roots where at least one is short. Then again using part (i), the result concludes.
\qed

For $\b\in R^\times$, we consider automorphism
$\Phi_\b$ of $\fg$ defined by
$$\Phi_\b=\exp(\ad x_\b)\exp(\ad -x_\b)\exp(\ad x_\b),$$
which satisfies $\Phi_\b(\fg_\a)=\fg_{w_\b(\a)}$ for $\a\in R$, see \cite[Proposition I.1.27]{AABGP97}. 
 
\begin{lem}\label{lemassump0} Assume that the extended affine Lie algebra $\fg$  satisfies either of the following:
 
(i) $\hbox{rank }\fg>1$,
 
(ii) $\fg$ is finite dimensional simple or an affine Lie algebra.
 
 \noindent
Then
\begin{equation}\label{assump1}
[x_\b,x_\d^\a]\in2\bbbz x_{\b+\d},\quad(\a,\b\in R^\times,\;0\not=\d\in R^0,\;\a=\pm\b\;\mod\; \v^0).
\end{equation}
\end{lem}

\proof 
(i) Assume that $\fg$ has rank $>1$, $\a,\b\in R^\times$ and  $\b=\a+\eta$ for some $\eta\in R^0$. We have $[x_\b,x^\a_\d]=cx_{\b+\d}$, for some complex constant $c$.  We must show that $c\in 2\bbbz$.
Since rank $\fg>1$, one checks directly from realization of finite root systems that there exists $\gamma\in R^\times $ such that $\gamma+\b+\d=\gamma+\a+\eta+\d\in R^\times$ and $\gamma\pm\a\not\in\v^0$.
Then we have
\begin{equation}\label{temp9}
[x_\gamma,[x_\b,[x_{\a+\d},x_{-\a}]]]=[x_\gamma, cx_{\b+\d}]=c c_{\gamma,\b+\d} x_{\b+\gamma+\d},
\end{equation}
where $cc_{\gamma,\b+\d}$ is non-zero. On the other hand, using the Jacobi identity we see that
\begin{eqnarray*}
[x_\gamma,[x_\b,[x_{\a+\d},x_{-\a}]]]
&=&-[[x_{\a+\d},x_{-\a}],[x_{\gamma}, x_{\a+\eta}]]
+[x_{\a+\eta},[x_\gamma,[x_{\a+\d},x_{-\a}]]]\\
&=&
-c_{\gamma,\a+\eta}[[x_{\a+\d},x_{-\a}],x_{\gamma'}]
+[x_{\a+\eta},[x_\gamma,[x_{\a+\d},x_{-\a}]]]\\
\end{eqnarray*}
where $\gamma'=\gamma+\a+\eta.$ Now we consider finite root systems
$$\dot{R}_1:=(\bbbz\a+\bbbz\gamma')\cap R,\qquad\dot{R}_2:=(\bbbz\a+\bbbz\gamma)\cap R,$$ and affine Lie algebras $\fg^a_{{\dot R}_1,\d}$, $\fg^a_{{\dot R}_2,\d}$, see Lemma \ref{newlem34}. The Chevalley automorphism $\tau$
then restricts to a Chevalley automorphism on $\fg^a_{{\dot R}_1,\d}$ and $\fg^a_{{\dot R}_2,\d}$.
Thus using realizations of affine Lie algebras as in \cite{Mit85}, we see that \begin{equation}\label{temp99}
[[x_{\a+\d},x_{-\a}],x_{\gamma'}]=kx_{\gamma'+\d}\andd [x_\gamma,[x_{\a+\d},x_{-\a}]]=k'x_{\gamma+\d}
\end{equation}
for some $k,k'\in\bbbz$, such that $k+k'\in2\bbbz$.
Then combining with (\ref{temp9}), we get
\begin{eqnarray*}
cc_{\gamma,\b+\d}x_{\gamma+\beta+\d}&=&
-kc_{\gamma,\a+\eta}x_{\gamma'+\d}+k'[x_{\a+\eta},x_{\gamma+\d}]\\
&=&
-kc_{\gamma,\b}x_{\gamma+\b+\d}+k'c_{\b,\gamma+\d}x_{\gamma+\b+\d}.
\end{eqnarray*}
Since $c_{\gamma,\b+\d},
c_{\gamma,\b},c_{\b,\gamma+\d}$ are equal up to $\pm$ signs, we get
$c\in2\bbbz$ as required.

(ii) By (i), we may assume $\hbox{rank } \fg=1$. If $\fg$ is finite dimensional simple, then $\b=\pm\a$ and the result is immediate. If $\fg$ is affine, then one can check from the realizations of affine Lie algebras that the result holds. \qed

\begin{rem}\label{roottemp}
The existence of root $\gamma$ and the integers $k,k'$ with $k+k'\in 2\bbbz$ involved  in the proof of Lemma \ref{lemassump0}
is presented in the following table, where for each case we have given a typical choice. Assume that $\a=\dot\a+\d_\a$ for some $\dot\a\in\rd$ and $\d_\a\in R^0$. Type $C_\ell$ is excluded because
of duality with type $B_\ell$. Types $E_{6,7}$ are similar to type $E_8$. In the table below, for the finite root system $\dot{R}$ we have used the realization given in \cite[\S 12.1]{Hum80}, also the integers $k$ and $k'$ are computed using the Chevalley basis given in \cite{Mit85} for an affine Lie algebra.

\begin{center}
\begin{tabular}{c|c|c|c|c|}

\hline
Type &$\dot\a$ & $\gamma$ & $k$ & $k'$\\
\hline
$A_\ell$ &$\ep_i-\ep_j$ & $\ep_j-\ep_k,\;k\not=i$ & $1$ & $1$\\
\hline
$B_\ell$ &$\ep_i$ & $\ep_j-\ep_i$ & $0$ & $2$\\
&$\ep_i-\ep_j$&$\ep_j$ & $1$ & $1$\\
\hline
$D_\ell$ &$\ep_i+\ep_j$ & $-\ep_j+\ep_k,\;k\not=i$ & $1$ & $1$\\
\hline
$F_4$ & $\ep_i$ &$\ep_j-\ep_i$ & $0$ & $2$\\
& $\ep_i+\ep_j$ & $-\ep_j$ & $1$ & $1$\\
& $\frac{1}{2}\sum_{i=1}^4\ep_i$ & $-\ep_1$ & $1$ & $1$\\
\hline
$G_2$ & $\ep_1-\ep_2$ & $\ep_2-\ep_3$ & $1$ & $1$\\
& $2\ep_1-\ep_2-\ep_3$ & $-\ep_1-\ep_2+2\ep_3$ & $1$ & $1$\\
\hline
$E_8$ & $\ep_i-\ep_j$ & $\ep_j-\ep_k,\;k\not=i$ & $1$ & $1$\\
& $\frac{1}{2}\sum_{i=1}^8\ep_i$ & $-\ep_1-\ep_2$ & $1$ & $1$\\
\hline
\end{tabular}
\end{center}
\end{rem}

}

 \section{Integral structure of the core}\label{Integral structure of the core}\setcounter{equation}{0}
In this section, we use the results  given in Section \ref{Preliminary results on integral structure} and the concept of a reflectable base to give a $\bbbz$-form for the core of an extended affine Lie algebra equipped with a Chevalley automorphism. This leads to a a notion of Chevalley basis for the core. As a byproduct, we give an upper bound for the dimension of the isotropic root spaces of the core.   

  We continue with the assumptions and notations as in  Section \ref{Preliminary results on integral structure}.
In particular,  $(\fg,\fm,\fh)$ is a tame reduced 
extended affine Lie algebra with root system $R$, equipped with a Chevalley automorphism $\tau$. In this section, we further assume that  {\it $\fg$ has rank $>1$.}

 \begin{lem}\label{lemassump2}
 Let $\phi:=\sum_{k=0}^n\frac{(\ad x_\b)^k}{k!}$, where $\b\in R^\times$. Then for $\a\in R^\times$ and $\d\in R^0$, the following hold:
 
 (i) If $\a+\b\in R^\times$, then $\phi (x_\a)\in\span_\bbbz\{x_{\a-d\b},\ldots,x_{\a+u\b}\}$.

 (ii)  If $\a+\b\in R^0$, then $\phi(x_\a)\in\span_\bbbz\{x_\a,x_{\a+\b}^{-\a},x_{\a+2\b}\}$.
 
 (iii)  $\phi (x_\d^\a)\in\span_\bbbz\{x_\d^\a, x_{\b+\d}\}.$
 
 
 (iv) $\phi(h_\a)\in\span_\bbbz\{h_\a,x_\b\}$, where $h_\a:=[x_\a,x_{-\a}].$
  \end{lem} 
  
  \proof (i) If $\a+\b\in R^\times$, namely $\{\a,\b\}$ is a nilpotent pair, then the statement is a relation inside the  finite dimensional simple Lie algebra  $\fg_{\a,\b}$ (see (\ref{simple1})). Thus the result follows from the classical finite dimensional theory, see \cite{Hum80}. 
  
  (ii) If  $\a+\b\in R^0$, then $\a=-\b+\lam$  for some $\lam\in R^0$, and so
$$\phi (x_\a)=
\sum_{k=0}^2\frac{(\ad x_\b)^k}{k!} (x_\a)=x_\a+
x_{\lam}^{-\a}+\frac{1}{2}[x_\b,x_{\lam}^{-\a}].$$
Now the result follows from Lemma \ref{lemassump0}.

(iii) If $\a+\b\in\v^0$ or $\a-\b\in\v^0$, then we are done by Lemma \ref{lemassump0}. Assume now that $\a\pm\b\not\in \v^0$.
Then $\phi(x^\a_\d)=x_\d^\a+[x_\b,x_\d^\a]$. Moreover,
\begin{equation}\label{eq17}
[x_\b,x^\a_\d]=[x_\b,[x_{\a+\d},x_{-\a}]]=-\stackrel{A}{\overbrace{[x_{-\a},[x_\b,x_{\a+\d}]]}}-\stackrel{B}{\overbrace{[x_{\a+\d}
,[x_{-\a},x_\b]]}}.
\end{equation}
If $A\not=0$, then
\begin{equation}\label{eq18}
A=[x_{-\a},c_{\b,\a+\d}x_{\b+\a+\d}]=c_{\b,\a+\d}c_{-\a,\b+\a+\d}
x_{\b+\d},
\end{equation}
and if $B\not=0$, then
\begin{equation}\label{eq19}
B=[x_{\a+\d},c_{-\a,\b}x_{-\a+\b}]=c_{\a+\d,-\a+\b}c_{-\a,\b}x_{\b+\d}
\end{equation}
as required. 

(iv) We have $\phi(h_\a)=h_\a+
[x_\b,h_\a]=h_\a-\b(h_\a)x_\b$.

\qed

We set
$$\fg_c^\bbbz:=\span_\bbbz\{x_\a,x_\d^\a, h_\a\mid\a\in R^\times,0\not=\d\in R^0,\a+\d\in R\}.$$
We note that 
$$\fg_c\cong\bbbc\otimes_\bbbz\fg_c^\bbbz,$$
as isomorphism of vector spaces.

\begin{cor}\label{co1-1}
For each $\b\in R^\times$,
the automorphism $\Phi_\b$ stabilizes $\fg^\bbbz_c$.
\end{cor}

\proof It immediately follows from Lemma \ref{lemassump2}.\qed

\begin{pro}\label{protemp2} 
Let $\Pi$ be a reflectable base for $R$.

(i) 
For $0\not=\d\in R^0$, 
$$\fg_\d\cap\fg_c^\bbbz=\span_\bbbz\{ x_\d^\b\mid \b\in\Pi,\b+\d\in R\}\andd \dim_\bbbc (\fg_\d\cap\fg_c)\leq |\Pi|.$$

(ii) $\fg_c^\bbbz=\span_\bbbz\{x_\a,t_\b,x_\d^\b\mid\a\in R^\times,\d\in R^0,\b\in\Pi\}.$
\end{pro}

\proof (i) We must show that for $\a\in R^\times$,  $x_\d^\a\in\span_\bbbz\{x_\d^\b\mid\b\in\Pi\}$. 
Now for $\a\in R^\times$, we have $\a=w_{\b_1}\cdots w_{\b_{k-1}}(\b_k)$ for some $\b_1,\ldots,\b_k\in\Pi$. Then taking $w=w_{\b_1}\cdots w_{\b_{k-1}}$, we get
$$x_\d^\a=x_\d^{w(\b_k)}\in\bbbz\Phi_{\b_1}\cdots\Phi_{\b_{k-1}}(x_\d^{\b_{k}}).$$ 
Then the fact that automorphisms $\Phi_{\b_i}$ preserve the root space $\fg_\d$ together with Lemma \ref{lemassump2} gives the result.

(ii) Since $\la\Pi\ra=\la R\ra$, for $\a\in R^\times$, we have
$ t_\a\in\la t_\b\mid\b\in\Pi\ra$. This together with (i) now completes the proof.  \qed

\begin{rem}\label{remtemp1}
Since reflectable bases are characterized in \cite{AYY12,ASTY19}, Proposition \ref{protemp2} provides an effective way of computing dimensions of isotopic root spaces of the core; compare with \cite[Proposition 1.4.2]{ABFP09} which gives an upper bound for the dimension of the isotropic root spaces of the core modulo center.
\end{rem}

\begin{thm}\label{thmnew5}
Let $(\fg,\fm,\fh)$ be a tame reduced extended affine Lie algebra with root system $R$ of rank $>1$, admitting a Chevalley automorphism $\tau$. Then $\fg_c^\bbbz$ is a $\bbbz$-Lie algebra. In fact $\fg_c^\bbbz$ is a $\bbbz$-form of $\fg_c$, generated as a $\bbbz$-Lie algebra by {the }{Chevalley system $\mathcal{C}=\{x_{\a}\in\fg_\a\mid\a\in R^\times\}$.}
\end{thm}

\proof By Proposition \ref{protemp2}(ii), $\fg_c^\bbbz=\span_{\bbbz}\{ x_\a,t_\b,x_\d^\b\mid\a\in R^\times,\d\in R^0,\b\in\Pi\}.$
Thus considering Lemma \ref{lemassump2}, we only need to show that $[x_\d^\a,x_\eta^\b]\in\fg_c^\bbbz$ for $\a,\b\in R^\times$ and $\d,\eta\in R^0$. Now
\begin{eqnarray*}
  [x_\d^\a,x_\eta^\b]&=&
  [[x_{\a+\d},x_{-\a}],x_\eta^\b]\\
  &=&-[[x_\eta^\b,x_{\a+\d}], x_{-\a}]-
[[x_{-\a},x_{\eta}^\b], x_{\a+\d}]\\
(\hbox{Lemma \ref{lemassump2}(iii)})&=&
k[x_{\a+\d+\eta},x_{-\a}]+k'[x_{-\a+\eta},x_{\a+\d}]\\
&=&kx_{\d+\eta}^\a+k'x_{\d+\eta}^{-\a-\d}\\
(\hbox{Proposition \ref{protemp2}}(i))&\in&\fg_c^\bbbz.
\end{eqnarray*}
\qed


\begin{thm}\label{thmnew1}({\bf Integral Structure})
Let $(\fg,\fm,\fh)$ be a tame reduced extended affine Lie algebra of rank $>1$ with root system $R$,  admitting a Chevalley automorphism $\tau$. Let $\dot\Pi=\{\a_1,\ldots,\a_\ell\}$ be a base of $\rd$, and fix a reflectable base $\Pi$ for $R$ such that $\dot\Pi\sub \Pi$.  Then, 
$\fg_c$ admits a basis, 
$${\mathcal B}=\{x_\a\in\fg_\a\mid\a\in R^\times\}\cup
\{h_i,c_j\in\fh\mid 1\leq i\leq\ell, 1\leq j\leq\nu\}\cup\big(\bigcup_{\d\in R^0{\setminus\{0\}}}{\mathcal X}_\d\big),$$
where ${\mathcal X}_\d\sub\{x^\a_\d\mid\a\in \Pi,\;\a+\d\in R\}$ is a basis of $\fg_\d\cap\fg_c$ such that

(i) $[c_i,\fg]=0,$ for all $i$,

(ii) $[h_i, h_j]=0$, for all $i,j$,

(iii) $[h_i, x_\a]=(\a,\a_i^\vee)x_\a,$ $\a\in R$, $1\leq i\leq\ell$,


(iv) $[x_\a,x_{-\a}]\in\sum_{i=1}^\ell\bbbz h_i+\sum_{j=1}^\nu\bbbz c_j,$ $\a\in R^\times,$

(v) $[x_\a,x_\b]=c_{\a,\b}x_{\a+\b}$, where
$c_{\a,\b}=0$ if $\a+\b\not\in R$, and if $\{\a,\b\}$ is nilpotent then 
\begin{equation}\label{form2}
c_{\a,\b}=-c_{-\a,-\b}\andd
c_{\a,\b}^2=\frac{u(d+1)(\a+\b,\a+\b)}{(\b,\b)}.
\end{equation}
 In particular,
$c_{\a,\b}\in\bbbz$.

(vi)  $[x_\a,x_\d]\in\bbbz x_{\a+\d}$, $\a\in R^\times$, $\d\in R^0$ and $x_\d\in\xx_\d$,

(vii)  $[x_\b,x_{-\b+\d}],[x_\eta,x_\mu]\in\span_\bbbz\{x^\a_{\d}\mid
\a\in\Pi,\;\a+\d\in R\}$, $\b\in R^\times$,  $\eta,\mu\in R^0$, $x_\eta\in\fg_\eta$, $x_\mu\in\fg_\mu$, $\eta+\mu=\d$.

\end{thm}

\proof {Let $\mathcal{C}=\{x_{\a}\in\fg_\a\mid\a\in R^\times\}$ be a Chevalley system for $(\fg,\tau)$.} Since non-isotropic root spaces are $1$-dimensional, {the system} {$\mathcal{C}$ forms} a basis for the subspace of $\fg$ consisting of non-isotropic root spaces. We next find a basis for $\fg_c\cap\fg_0$. 
We know that 
 $\fg_0\cap\fg_c$ is spanned by $[\fg_\a,\fg_{-\a}]=\bbbc t_\a$, $\a\in R^\times$. Now for $\a\in R^\times$ we have
 $\a\in\span_{\bbbz}\{\a_1,\ldots,\a_\ell,\sg_1,\ldots,\sg_\nu\}$ (see \ref{sen1}) and so $t_\a$ is spanned by $t_{\a_i}, t_{\sg_j}$, $1\leq i\leq\ell$,
 $1\leq j\leq\nu$. So the elements
 $h_i:=[x_{\a_i},x_{-\a_i}]=2t_{\a_i}/(\a_i,\a_i)$, $1\leq i\leq\ell,$ and
$$c_j=\left\{
\begin{array}{ll}
t_{\sg_j}&\hbox{if }1\leq j\leq t\\
{2t_{\sg_j}}/{(\b,\b)}&\hbox{ if }t+1\leq j\leq\nu,
\end{array}\right.
$$
where $\b$ is a root of maximum length, form a basis for $\fg_c\cap\fg_0$ (here $t$ is the twist number of the root system, introduced prior to (\ref{eq99-1})).

Next, using Proposition \ref{protemp2}, we consider a basis
$\xx_\d\sub\{x_\d^\a\mid\a\in\Pi,\;\a+\d\in R\}$ for $\fg_\d\cap\fg_c.$
Thus the set ${\mathcal B}$ given in the statement is a $\bbbc$-basis for $\fg_c$. In what follows we show that the basis ${\mathcal B}$ satisfies properties (i)-(vii) of the statement, where we already know from Proposition \ref{hump3} and Theorem \ref{thmnew5} that parts (v)-(vii) hold.

We have for $\a\in R$,  $[t_{\sg_i},\fg_\a]=\a(t_{\sg_i})\fg_\a=\{0\}$ as $(\a,\sg_i)=0$. Thus $[c_i,\fg]=0$. Since $\fh $ is abelian, we have $[h_i,h_j]=0$. Next for $\a\in R^\times$, 
$$[h_i,x_\a]=\a(h_i)x_\a=
\frac{2\a(t_{\a_i})}{(\a_i,\a_i)}x_\a=(\a,\a^\vee_i)x_\a.$$
Therefore properties (i)-(iii) hold.

 Next let  $\a\in R^\times$. We have $\a=\dot\a+\d$ for some
$\dot\a\in\rd$ and $\d\in R^0$. So using
(\ref{ned1}), we have $\d=\sum_{i=1}^t m_ik_\a\sg_i
+\sum_{i=t+1}^\nu m_i\sg_i$, where $m_i\in\bbbz$, $k_\a=1$ if $\a$ is short and $k_\a=k=(\a,\a)/2$  if $\a$ is long. Then
$$[x_\a,x_{-\a}]=h_\a=\frac{2t_\a}{(\a,\a)}=
\frac{2t_{\dot\a}}{(\a,\a)}+\frac{2t_\d}{(\a,\a)}=h_{\dot\a}+\frac{2t_\d}{(\a,\a)},$$ where
from the theory of finite dimensional simple Lie algebras we know that $h_{\dot\a}\in\sum_{i=1}^\ell\bbbz h_i$, and
\begin{eqnarray*}
\frac{2t_\d}{(\a,\a)}
&=&\frac{2}{(\a,\a)}(\sum_{i=1}^t m_i k_\a t_{\sg_i}
+\sum_{i=t+1}^\nu m_i t_{\sg_i})\\
&=&
\frac{2k_\a}{(\a,\a)}\sum_{i=1}^t m_i  t_{\sg_i}
+\frac{(\b,\b)}{(\a,\a)}\sum_{i=t+1}^\nu m_i\frac{2 t_{\sg_i}}{(\b,\b)}\\
&=&
\frac{2k_\a}{(\a,\a)}\sum_{i=1}^t m_i  c_i
+\frac{(\b,\b)}{(\a,\a)}\sum_{i=t+1}^\nu m_i c_i\in\sum_{i=1}^\nu\bbbz c_i.
\end{eqnarray*}
This completes the proof of (iv).\qed

 We conclude this section with a discussion about dimensions of the isotropic root spaces of the core. To proceed, we need to introduce some terminologies.  We first recall some facts from Section \ref{sub-sec-EALA}.  The root system $R$ involves two semilattices $S$ and $L$. If $\Lam=\la R^0\ra$, then $S$ contains a $\bbbz$-basis $\sg_1,\ldots,\sg_\nu$ of $\Lam$. Moreover,  if $X$ is non-simply laced then $\sg_1,\ldots,\sg_\nu$ can be chosen such that
$$\la L\ra=k\Lam_1\oplus\Lam_2\hbox{ with }\Lam_1=k\bbbz\sg_1\oplus\cdots\oplus k\bbbz\sg_t\hbox{ and }\Lam_2=\bbbz\sg_{t+1}\oplus\cdots\oplus\bbbz\sg_\nu,$$
where $t$ is the twist number of $R$, $k=3$ if $X=G_2$ and $k=2$ otherwise.
  To say that $S$ is a semilattice means that
$$S=\cup_{i=0}^m(\tau_i+2\Lam)\hbox{ where }\tau_0=0\hbox{ and }\tau_i\not=\tau_j\hbox{ mod }2\Lam, \;(i\not=j).
$$ 
If the type of $R$ is non-simply laced then $S=S+L$ and $L=kS+L$, where $k=3$ if $X=G_2$ and $k=2$ otherwise.  Furthermore, there exist two semilattices
$$
\begin{array}{c}
S_1=\cup_{i=1}^{m_1}(\gamma_i+2\la \Lam_1\ra)\andd
S_2=\cup_{i=1}^{m_2}(\eta_i+2\la \Lam_2\ra)\vspace{2mm}\\
\end{array}
$$ with $\gamma_i\not=\gamma_j$ mod $2\Lam_1$, $i\not=j$, and
$\eta_i\not=\eta_j$ mod $2\Lam_2$, $i\not=j$,
such that $S=S_1\oplus\Lam_2$ and $L=k\Lam_1\oplus S_2$.
The integers $m_1$ and $m_2$ are denoted by $\ind(S_1)$ and $\ind(S_2)$, respectively.

With respect to a fixed base $\dot\Pi=\{\a_1,\ldots,\a_\ell\}$ of $\rd$, let $\te_s$ and $\te_l$ be  any  short root and any long root of $\dot R$, respectively. If $R$ is of simply laced types, then $\te_s=\te_l$. Now depending on type $X$ of $R$, we introduce a set $\Pi(X)$ as follows.

\begin{tab}\label{tab11} Minimal reflectable bases
$$\hbox{
 $\;$\vspace{2mm} \\
{\begin{tabular}{|c|c|}
$\mbox{Type}$& $\Pi(X)$  \\
\hline
$A _1$ & $\{\a_1, \tau_1 -\a_1, \ldots, \tau_m-\a_1\}$  \\
\hline
$A_\ell (\ell>1)$ & $\{\a_1, \ldots, \a_\ell, \sg_1-\te_s, \ldots, \sg_\nu-\te_s\}$\\
\hline 
$D_\ell$ & $\{\a_1, \ldots, \a_\ell, \sg_1-\te_s, \ldots, \sg_\nu-\te_s\}$ \\
\hline
$E_{6,7,8}$ & $\{\a_1, \ldots, \a_\ell, \sg_1-\te_s, \ldots, \sg_\nu-\te_s\}$\\
\hline
$F_4$ & $\{\a_1, \dots, \a_\ell, \sg_1-\te_s, \cdots, \sg_t-\te_s, \sg_{t+1}-\te_l,
\ldots, \sg_\nu -\te_l\}$ \\
\hline
$G_2$ & $\{\a_1, \dots, \a_\ell, \sg_1-\te_s, \cdots, \sg_t-\te_s, \sg_{t+1}-\te_l,
\ldots, \sg_\nu -\te_l\}$\\
\hline
$B_2$ & $\{\a_1, \a_2, \gamma_1-\te_s, \ldots, \gamma_{m_1}-\te_s, 
\eta_1 -\te_l, \ldots, \eta_{m_2}-\te_l\}$\\
\hline
$B_\ell\; (\ell>2)$ & $\{\a_1, \ldots, \a_\ell, \gamma_1-\te_s, \ldots, \gamma_{m_1}-\te_s, \sg_{t+1}-\te_l, \ldots, \sg_\nu -\te_l \}$\\
\hline
$C_\ell \;(\ell>2)$ & $\{\a_1, \ldots, \a_\ell, \sg_1 -\te_s, \ldots, \sg_t -\te_s, \eta_1 -\te_l, \ldots, \eta_{m_2}-\te_l\}$\\
\hline
\end{tabular}}
}
$$
\end{tab}
By \cite{Az99}, $\Pi(X)$ is a reflectable base for $R$ with minimal cardinality.

\begin{thm}\label{thmdim}
Let rank $\fg>1$ and $0\not=\d\in R^0$. Then  an upper bound
for $\dim(\fg_c\cap\fg_\d)$ is $\ell+\nu$ if $X$ is simply laced, and for non-simply laced types is given by
$$\hbox{
 $\;$\vspace{2mm} \\
{\begin{tabular}{l|c|c|c|c}
& $B_\ell\; (\ell\geq 3)$&$C_\ell\;(\ell\geq 3)$&$F_4$&$G_2$  \\
\hline
$\d\in L$ & $\ell+\ind(S_1)+(\nu-t)$ &$\ell+t+1$&$4+\nu$ &$2+\nu$\\
\hline
$\d\not\in L$ & $2$&$\ell+t+\ind(S_2)-2$&$2+t$&$1+t$\\
\hline 
\end{tabular}}
}
$$
If $X=B_2$ and $\d=\lam_1+\lam_2$, $\lam_1\in\Lam_1$, $\lam_2\in\Lam_2$, then the upper bound is given by
$$\left\{\begin{array}{l}
\ell+\ind(S_1)+(\nu-t)\quad\hbox{ if }\d\in L,\\
4\quad\hbox{ if }\lam_1\in 2\Lam_1, \lam_2=\eta_j+\eta_k\hbox{ mod }2\Lam_2, j\not=k,\\
2\quad\hbox{ otherwise.}
\end{array}\right.
$$
\end{thm}

\proof We apply the reflectable base $\Pi(X)$ given in  Table \ref{tab11} to
{Proposition} \ref{protemp2}(i), namely in the spanning set $\{x_\d^\a\mid \a\in\Pi(X),\;\a+\d\in R\}$ for $\fg_\d\cap\fg_c$ we shall count the number of  $\a\in\Pi(X)$ such that $\a+\d\in R$.

Let $\d\in R^0\sub\Lam$. If $X$ is simply laced 
then $S$ is a lattice and so for each $\a\in\Pi(X)$, we have $\a+\d\in R$. Thus
  $\dim(\fg_c\cap\fg_\d)\leq|\Pi(X)|=\ell+\nu$.

Next, we consider $X$ to be of type $B_\ell$, $\ell\geq 3$. We have $R^0=(R^0\setminus S)\uplus (S\setminus L)\uplus L$.
Assume first that  $\d\in L$. Then for $\a\in\dot\Pi$, we have
$\dot\a+\d\sub R$. Also for $1\leq i\leq m_1$, 
$\gamma_i-\te_s+\d\sub-\te_s+S+L\sub-\te_s+S\sub R$.
Finally, since in the case under consideration $L$ is a lattice  we have for $t+1\leq j\leq \nu$, 
$\sg_i-\te_l+\d\sub\te_l+L\sub R$. Thus $\dim(\fg_c\cap\fg_\d)\leq
|\dot\Pi|+m_1+(\nu-t)$, as required.

If $\d\in S\setminus L$, then $\d=\lam_1+\lam_2$, where
$\lam_1=\gamma_i+2\lam$ for some $1\leq i\leq m_1$ and $\lam\in\Lam_1$. Without loose of generality assume that
$\a_1,\ldots,{\a_{\ell-1}}$ are long and $\a_\ell$ is short.
Then for $1\leq i\leq\ell-1$, {$\a_i+\d
\not\in R$,} 
and $\a_\ell+\gamma_i+2\lam+\lam_2
\in\rds+S_1+\Lam_2=\rds+S\sub R$. Also
for $1\leq j\leq m_1$, {we have
$\gamma_i-\theta_s+\d=\gamma_j-\te_s+\gamma_i+2\lam+\lam_2
\in\rds+\gamma_j+\gamma_i+2\Lam_1+\Lam_2$.} But the latter is contained in $R$ only if $i=j$. Finally for
$t+1\leq j\leq\nu$, 
$\sg_j-\theta_l+\d=\sg_j-\te_l+\gamma_i+2\lam+\lam_2\in\rdl+\gamma_i+2\Lam_1+\Lam_2$. But the latter has empty intersection with $R$. Combining all these informations, we get
$\dim(\fg_c\cap\fg_\d)\leq 2.$

Finally, assume that $\d\in R^0\setminus S$. Since
$S=S_1\oplus\Lam_1$, and $R^0=S+S=(S_1+S_1)\oplus\Lam_2$, we get
$\d=(\gamma_i+\gamma_j+2\lam_1)+\lam_2$ for some $1\leq i\not=j\leq m_1$, $\lam_1\in2\Lam_1$ and $\lam_2\in\Lam_2$. {Since  $\d\not\in S$, we have
$\a+\d\not\in R$ for all $\a\in\dot\Pi$.} Now for $1\leq k\leq m_1$,
$\gamma_k-\te_s+\d=(\gamma_k-\gamma_i-\gamma_j+2\lam_1)-\te_s$ which belong to $R$ if and only if $k=i$ or $k=j$. Finally
for $t+1\leq k\leq\nu$, we have
$\sg_k+\d\not\in L$ and so $\sg_k-\te_l+\d\not\in R$. Thus
$\dim(\fg_\d\cap\fg_c)\leq 2$ as it was claimed. 

The proof for remaining types follows by a similar approach.\qed

As one knows (see \cite[Corollary 8.3]{Kac85}), Theorem \ref{thmdim} can be sharpen for affine Lie algebras, we illustrate it here.
\begin{cor}\label{affinedim}
Suppose that $\fg>1$, $\nu=1$ and $0\not=\d\in R^0$.  Then $\dim(\fg_\d)=\ell$ if $X$ is simply laced, and if  $X$ is non-simply laced then  $dim(\fg_\d)$ is given by the following table:
$$\hbox{
 $\;$\vspace{2mm} \\
{\begin{tabular}{l|c|c|c|c}
& $B_\ell$&$C_\ell$&$F_4$&$G_2$  \\
\hline
$\d\in L$ & $\ell$ &$\ell$&$4$ &$2$\\
\hline
$\d\not\in L$ & $1$&$\ell-1$&$2$&$1$\\
\hline 
\end{tabular}}
}
$$
\end{cor}
 
 \proof If $\nu=1$, each reflectable base given in Table \ref{tab11}, is of the form $\Pi(X)=\{\a_1,\ldots,\a_\ell,\sg-\b\}$
  for some isotropic root $\sg$ and some $\b\in\dot R$. Since, we are in affine case, there exists an isotropic root $\d_1$ such that
  $\d=k\d_1$ and $\sg=k'\d_1$ for some $k,k'\in\bbbz$.
  Then from realizations of affine Lie algebras we know that 
  $x_\d^{\sg-\b}=x_{k\d_1}^{k'\d_1-\b}\in\bbbc x_\d^{-\b}$. Now as $\b=w_{\b_1}\cdots w_{\b_{t}}(\b_{t+1})$ for some $\b_i\in\dot\Pi$, we get
  $x_\d^{-\b}\in\bbbz\Phi_{\b_1}\cdots\Phi_{\b_t}(x_{\d}^{-\b_{t+1}})\in\span_{\bbbz}\{x_\d^{\a_1},\ldots,x_\d^{\a_\ell}\}.$ Thus using {Proposition} \ref{protemp2}(i), it remains to count the number
  of $\a_i$ such that $\a_i+\d\in\ R$, $1\leq i\leq\ell$. But this is immediate using description (\ref{sss}) of $R$.\qed
   
    \section{On the Uniqueness of integral  structure}\label{On the Uniqueness of integral structure}
The goal  in this section is to show that the integral structure provided for the core of an extended affine Lie algebra given in Section \ref{Integral structure of the core} is independent of the choice of the {Chevalley system.}

 We continue with the notations and terminologies as in Sections \ref{Preliminary results on integral structure} and \ref{Integral structure of the core}. As before $(\fg,\fm,\fh)$ is a tame reduced extended affine Lie algebra of rank $>1$, with root system $R$. By Theorem \ref{thmnew5}, $\fg^\bbbz_c$ is generated as a $\bbbz$-Lie algebra
 by {the} choice of {a Chvalley system} $\{x_\a\mid\a\in R^\times\}$; we denote this dependence by $\fg_c^\bbbz=\fg_c^\bbbz(x_\a\mid\a\in R^\times)$. In this section, we discuss relations between $\fg^\bbbz_c(x_\a\mid\a\in R^\times)$ and
 $\fg^\bbbz_c(\bar{x}_\a\mid\a\in R^\times)$ for two choices of {Chevalley systems} $\{x_\a\}_{\a\in R^\times}$ and 
 $\{\bar{x}_\a\}_{\a\in R^\times}$. We have $\bar{x}_\a=\eta_\a x_\a$, $\a\in R^\times$, for some $\eta_\a\in\bbbc$.
 Since $[x_\a,x_{-\a}]=h_\a=[\bar{x}_\a,\bar{x}_{-\a}]$, we get
 $\eta_\a\eta_{-\a}=1$ for each $\a\in R^\times$. We recall {from Proposition \ref{hump3} part (ii)} that for a nilpotent pair $\{\a,\b\}$ of roots in $R$, we have
 $[x_\a,x_\b]=c_{\a,\b}x_{\a+\b}$, $[\bar{x}_\a,\bar{x}_\b]=\bar{c}_{\a,\b}\bar{x}_{\a+\b}$ {and $c_{\a,\b}=\pm\bar{c}_{\a,\b}$}.

 \begin{lem}\label{ned3} 
(i) If $\{\a,\b\}$ is a nilpotent pair  in $R$ then
$c_{\a,\b}\eta_{\a+\b}=\bar{c}_{\a,\b}\eta_{\a}\eta_{\b}$
{and}
$\eta_{\a+\b}=\pm \eta_\a\eta_{\b}$.

(ii) Suppose  $\a,\b\in R^\times$, $\d\in R^0$, $\a+\d,\b+\d\in R$.
Then 
 $\eta_{\a+\d}\eta_{-\a}=\pm\eta_{\b+\d}\eta_{-\beta}$.
  \end{lem}
  
  \proof (i) By \cite[Lemma 3.5]{AP19}, $R_{\a,\b}=(\bbbz\a\oplus\bbbz\b)\cap R$
  is an irreducible reduced finite root system. Then \cite[\S 25.3]{Hum80} gives the result.
  
  (ii)  Assume first that $\a$ and $\b$ have the same length.
  Let $\dot\theta$ be the highest long (short)  root of $\dot R$
  if $\a$ is long (short). We have $\a=\dot\a+\sg$ for some
  $\dot\a\in\dot R$ and $\sg\in R^0$. We may assume without loos
  of generality that $\dot\a\in\rd^+$. We set
  $\theta:=\dot\theta+\sg$. Since $\dot\a$ and $\dot\theta$ have the same length,  we have $\theta\in R$. It is enough to show that 
  $\eta_{\a+\d}\eta_{-\a}=\eta_{\theta+\d}\eta_{-\theta}$. 
  Assume $\dot\a\not=\dot\theta$.
  Since
  $\theta+\dot\a\not\in\dot R$, we get from the Jacobi identity that
  \begin{eqnarray*}
  [\bar{x}_\theta, [\bar{x}_{\a+\d},\bar{x}_{-\a}]]&=&-[\bar{x}_{\a+\d},[\bar{x}_{-\a},\bar{x}_\theta]]\\
  &=&
  -\bar{c}_{-\a,\theta}\bar{c}_{\a+\d,\theta-\a}\bar{x}_{\theta+\d}
  =-\eta_{\theta+\d}
  \bar{c}_{-\a,\theta}\bar{c}_{\a+\d,\theta-\a}x_{\theta+\d}
\end{eqnarray*}  
On the other hand
 \begin{eqnarray*}
  [\bar{x}_\theta, [\bar{x}_{\a+\d},\bar{x}_{-\a}]]&=&\eta_\theta\eta_{\a+\d}
  \eta_{-\a}[x_\theta, [x_{\a+\d},x_{-\a}]]\\
  &=&
  -\eta_\theta\eta_{\a+\d}
  \eta_{-\a}[x_{\a+\d},[x_{-\a},x_\theta]]\\
  &=&
 - c_{-\a,\theta}c_{\a+\d,\theta-\a}\eta_\theta\eta_{\a+\d}
  \eta_{-\a}x_{\theta+\d}.
\end{eqnarray*}  
  Comparing the above equalities we get,
  $$\eta_{\theta+\d}=\pm\eta_\theta\eta_{\a+\d}
  \eta_{-\a}$$
Since $\eta_{\theta}\eta_{-\theta}=1$, we are done.

To complete the proof, it remains to show that if $R$ is not simply laced, and $\b\in R$ is long, then
$\eta_{\b+\d}\eta_{-\b}=\eta_{\a+\d}\eta_{-\a}$ for some short root $\a$. So assume that $\b$ is long.
We have $\b=\dot\b+\sg$ for some long root $\dot\b\in\rd$ and $\sg\in R^0$. From realization of finite root systems, one can see that $\dot\b=\dot{\a}_1+\dot{\a}_2$ for some short roots 
$\dot{\a}_1$ and $\dot{\a}_2\in\dot R$. Then we have from part (i),
\begin{eqnarray*}
\eta_{\b+\d}\eta_{-\b}&=&
\eta_{{\dot\a}_1+\dot{\a}_2+\sg+\d}\eta_{-\dot{\a}_1-\dot{\a}_2-\sg}\\
&=&
\eta_{({\dot\a}_1+\d)+(\dot{\a}_2+\sg)}\eta_{-\dot{\a}_1-(\dot{\a}_2+\sg)}\\
&=&
\pm\eta_{{\dot\a}_1+\d}\eta_{-\dot{\a}_1}\underbrace{\eta_{\dot{\a}_2+\sg}\eta_{-\dot{\a}_2-\sg}}_1\\
&=&
\pm\eta_{{\dot\a}_1+\d}\eta_{-\dot{\a}_1}.
\end{eqnarray*} 
\qed

\begin{thm}\label{simply} ({\bf Uniqueness Theorem})
Assume that $\fg$ is a reduced tame extended affine Lie algebra of rank $>1$. Then, as $\bbbz$-Lie algebras we have
$\fg_c^\bbbz(x_\a\mid\a\in R^\times)\cong\fg_c^\bbbz(\bar{x}_\a\mid\a\in R^\times)$, for two choices of {Chevalley systems} $\{x_\a\}_{\a\in R^\times}$ and
$\{\bar{x}_\a\}_{\a\in R^\times}$.
\end{thm}
  \proof
 In what follows we always assume that $1\leq i\leq\ell$, $1\leq j\leq\nu$, $\a\in R^\times$ and $\d\in R^0\setminus\{0\}$. From Theorem \ref{thmnew1}, we know that the set $\{x_\a,h_i,{c_j}\}$
is  a basis for the subspace
{$(\fg_0\cap\fg_c)\oplus\sum_{\a\in R^\times}\fg_\a$} of $\fg$.
We consider an assignment 
$$\fg_c^\bbbz(x_\a\mid\a\in R^\times)\stackrel{f}{\longrightarrow}\fg_c^\bbbz(\bar{x}_\a\mid\a\in R^\times)$$
$$\begin{array}{c}
h_i\mapsto h_i,\quad
c_j\mapsto c_j,\quad
x_\a\mapsto \eta_{\a}\bar{x}_\a,\\
\sum_{t=1}^nk_tx^{\a_t}_\d\mapsto \sum_{t=1}^nk_t\eta_{\a_t+\d}\eta_{-\a_t}\bar{x}^{\a_t}_\d,
\end{array}
$$
where $\bar{x}^{\a_t}_\d:=[\bar{x}_{\a_t+\d},\bar{x}_{-\a_t}]$, $\a_t\in R^\times$, $k_t\in\bbbz$ and $n\in\bbbz_>0$. 
The first three assignments induces a group homomorphism
$\span_{\bbbz}\{x_\a,h_i,c_j\}\stackrel{f}{\rightarrow}
\span_{\bbbz}\{\bar{x}_\a,h_i,c_j\}$. To extend $f$ to a group homomorphism from $\fg_c^\bbbz(x_\a,h_i,c_j)$ onto
$\fg_c^\bbbz(\bar{x}_\a,h_i,c_j)$, it is enough to show that
if $\sum_{t=1}^{n}k_tx^{\a_t}_\d=0$, $\a_t\in R^\times$,
then 
$\sum_{t=1}^nk_t\eta_{\a_t+\d}\eta_{-\a_t}\bar{x}^{\a_t}_\d=0$. But
$$\sum_{t=1}^nk_t\eta_{\a_t+\d}\eta_{-\a_t}\bar{x}^{\a_t}_\d=\sum_{t=1}^{n}k_t(\eta_{\a_t+\d}\eta_{-\a_t})^2x^{\a_t}_\d
=\eta\sum_{t=1}^{n}k_tx^{\a_t}_\d=0,$$
where $\eta=(\eta_{\a_t+\d}\eta_{-\a_t})^2$, for $1\leq t\leq n$, see Lemma \ref{ned3}.
This shows that $f$ is well defined, and  is a group homomorphism.

We now check that $f$ is a Lie algebra homomorphism. Let $\a,\b\in R^\times$ and $\d,\sg\in R^0$. Then by Lemma \ref{ned3} part (i),
\begin{eqnarray*}
f([x_\a,x_\b])&=&
c_{\a,\b}f(x_{\a+\b})=c_{\a,\b}\eta_{\a+\b}\bar{x}_{\a+\b}=\bar{c}_{\a,\b}\eta_{\a}\eta_{\b}\bar{x}_{\a+\b}\\
&=&
[\eta_{\a}\bar{x}_\a,\eta_{\b}\bar{x}_\b]=
[f(x_\a),f(x_\b)].
\end{eqnarray*}
Next, we consider brackets of the form
$[x_\b,x_\d^\a]$. First note that
$$[\bar{x}_\b,\bar{x}^\a_\d]=\eta_{\b}\eta_{\a+\d}\eta_{-\a}[x_\b,x_\d^\a]=c\eta_{\b}\eta_{\a+\d}\eta_{-\a}x_{\b+\d}=\bar{c}\eta_{\b+\d}x_{\b+\d},$$
where $c,\bar{c}\in\bbbz$, see Theorem \ref{thmnew1}(vi). On the other hand we get from Lemma \ref{ned3} part (ii) that $\eta_{\b+\d}=\pm\eta_{\b}\eta_{\a+\d}\eta_{-\a}$. This together with the last equality gives $c\eta_{\b+\d}=\bar{c}\eta_{\b}\eta_{\a+\d}\eta_{-\a}$. Then   
\begin{eqnarray*}
f([x_\b,x^\a_\d])&=&
cf(x_{\b+\d})=c\eta_{\b+\d}\bar{x}_{\b+\d}=\bar{c}\eta_{\b}\eta_{\a+\d}\eta_{-\a}\bar{x}_{\b+\d}\\
&=&
[\eta_{\b}\bar{x}_\b,\eta_{\a+\d}\eta_{-\a}\bar{x}^\a_\d]=[f(x_\b),f(x^\a_\d)].
\end{eqnarray*}
Finally for brackets of the form $[x^\a_\d,x^\b_\sg]$ we have
\begin{eqnarray*}
	f([x_\d^\a,x_\sg^\b])&=&
	f([[x_{\a+\d},x_{-\a}],x_\sg^\b])\\
	&=&f(-[[x_\sg^\b,x_{\a+\d}], x_{-\a}]-
	[[x_{-\a},x_{\sg}^\b], x_{\a+\d}])\\
	&=&-[[f(x_\sg^\b),f(x_{\a+\d})], f(x_{-\a})]-
	[[f(x_{-\a}),f(x_{\sg}^\b)], f(x_{\a+\d})]\\
	&=&\eta_{\b+\sg}\eta_{-\b}\eta_{\a+\d}\eta_{-\a}(-[[\bar{x}_\sg^\b,\bar{x}_{\a+\d}], \bar{x}_{-\a}]-
	[[\bar{x}_{-\a},\bar{x}_{\sg}^\b], \bar{x}_{\a+\d}])\\
	&=&\eta_{\b+\sg}\eta_{-\b}\eta_{\a+\d}\eta_{-\a}([[\bar{x}_{\a+\d},\bar{x}_{-\a}],\bar{x}_\sg^\b])\\
	&=&[\eta_{\a+\d}\eta_{-\a}\bar{x}^\a_\d,\eta_{\b+\sg}\eta_{-\b}\bar{x}^\b_\sg]=[f(x_\d^\a),f(x_\sg^\b)].
\end{eqnarray*} 
To check that  $f$ preserves the bracket for remaining terms is clear.
This completes the proof.
\qed

 We conclude this section with a remark concerning the extension of the $\bbbz$-form from the core to the entire ground extended affine Lie algebra.
 
 \begin{rem}\label{remned}
 Let $\fg$ be an extended affine Lie algebra with the core $\fg_c$ and the centerless core $\fg_{cc}:=\fg_c/{\mathcal Z}(\fg_c)$. In the theory of extended affine Lie algebras it is well understood that the essential structural features of the theory are encoded in $\fg_{cc}$.
Then one proceeds to promote the results from $\fg_{cc}$ to $\fg_c$ and then to $\fg$. A systematic study of relations between $\fg_c$, $\fg_{cc}$ and $\fg$ is given in \cite{Ne11}. We will consider this approach in a separate project to study integral structures of general extended affine Lie algebras. In Example \ref{exaned} below we consider a special case. 
\end{rem}
  
 \begin{exa}\label{exaned}
 Let $\dot\fg$ be a centerless Lie torus.\ For  this we mean 
 a triple $(\dot\fg,\fm,\dot\fh)$ consisting of a Lie algebra $\dot\fg$, a finite dimensional abelian subalgebra $\dot\fh$ and a non-degenerate invariant bilinear form $\fm$ on $\dot\fg$ satisfying certain conditions given in \cite[Chapter III.\S1]{AABGP97} {(see also \cite{Y06})}.\
 Without going through details, we just report here some of these conditions and their consequences which will be needed for our 
 purposes. \ As an $\dot\fh$-module, $\dot\fg$  has a weight space decomposition
 $\dot\fg=\sum_{\dot\a\in{\dot\fh}^\star}{\dot\fh}_{\dot\a}$ where ${\dot\fg}_{\dot\a}=\{x\in\dot\fg\mid [h,x]=\dot\a(h)x\hbox{ for all }h\in\dot\fh\}$.\
 The set $\dot R$  of weights of $\dot\fg$ is an irreducible finite root system.\
 As a Lie algebra $\dot\fg$ is provided by a $\bbbz^\nu$-grading $\dot\fg=\sum_{\sg\in\bbbz^\nu}{{\dot\fg}^\sg}$ such that
 ${\dot\fg}_{\dot\a}=\sum_{\sg\in\bbbz^\nu}({\dot\fg}^\sg\cap{\dot\fg}_{\dot\a})$, $\dot\a\in\dot R$.\
   
  Define $d_i\in\hbox{Der}(\dot\fg)$ by $d_i(x)=n_ix$, $x\in{\dot\fg}^{(n_1,\ldots,n_\nu)}$, and set ${\mathcal D}=\sum_{i=1}^\nu\bbbc d_i$. The ${\mathcal D}$ is a $\nu$-dimensional subspace of derivation algebra of $\dot\fg$. Let ${\mathcal C}=\sum_{i=1}^\nu\bbbc c_i$
be a $\nu$-dimensional vector space and set $\fg:=\dot\fg\oplus{\mathcal C}\oplus {\mathcal D}$ and $\fh:=\dot\fh\oplus{\mathcal C}\oplus{\mathcal D}$.  The bracket on $\dot\fg$ then extends to a Lie bracket $[\cdot,\cdot]'$ on $\fg$ such that 
$$\mathcal C\hbox{ is central, } 
[x,y]'=[x,y]+\sum_{i=1}^\nu(d_ix,y)c_i\hbox{ and }[d,x]'=d(x),\;
x,y\in\dot\fg,\; d\in{\mathcal D}.$$
 The form on $\dot\fg$ also extends in a natural way to a non-degenerate invariant form on $\fg$.
It turns out that $(\fg,\fm,\fh)$ is a tame extended affine Lie algebra of nullity $\nu$ with core $\fg_c=\dot\fg\oplus{\mathcal C}$. 

Let $\sg_1,\ldots,\sg_\nu$ be the dual basis of $\{d_1,\ldots,d_\nu\}$ with respect to the form. Then $\bbbz^\nu$ can be identified with $\sum_{i=1}^\nu\bbbz\sg_i\sub{\mathcal D}^\star$. The root system $R$ of $\fg$ satisfies
$R\sub\dot R+\bbbz^\nu\sub\dot\fh\oplus D^\star\sub\fh^\star$, where for $0\not=\a=\dot\a+\sg \in R$, $\dot\a\in\dot R$ and $\sg\in\bbbz^\nu$, 
$\fg_\a=\dot\fg_{\dot\a}\cap{\dot\fg}^\sg$ and
$\fg_0=\dot\fh\oplus{\mathcal C}\oplus{\mathcal D}$.

Now we consider a $\bbbz$-form $\fg^\bbbz_c=\fg^\bbbz_c(x_\a\mid\a\in R^\times)$ for 
$\fg_c$. Let 
$$\fg^\bbbz:=\fg_c^\bbbz\oplus\span_{\bbbz}\{d_i\mid 1\leq i\leq\nu\}.$$
If $\a=\dot\a+\sum_{i=1}^\nu k_i\sg_i\in R^\times$, $k_i\in\bbbz$,  then $[d_j,x_\a]=d_j(\sum_{i=1}^\nu k_i\sg_i)x_\a=k_jx_\a$.
It then follows that $\fg^\bbbz$ is a $\bbbz$-Lie algebra.
Since $\fg^\bbbz\otimes_\bbbz\bbbc\cong\fg$, $\fg^\bbbz$ is a 
$\bbbz$ form for $\fg$.
\end{exa}

  \section{chevalley automorphism for multi-loop realization}
\label{chevalley basis and multi-loop realization}
\setcounter{equation}{0}
In this section, we consider the multi-loop algebra based on an extended affine Lie algebra and the algebra of Laurent polynomials. We show that under suitable conditions on the involved automorphisms, the affinization procedure applied to the multi-loop algebra leads to an extended affine Lie algebra. We also investigate how a Chevalley automorphism on the underlying extended affine Lie algebra can be lifted to one on the multi-loop affinization. 

Throughout this section, we assume that $\aa$ is the algebra of Laurent polynomials in variables $x_i^{\pm},$ $1\leq i\leq \nu$, over the field of complex numbers. 

\begin{DEF}
	Let $(\fg,\fm,\fh)$ be an extended affine Lie algebra with root system $R$ and $\sg_1,\ldots,\sg_\nu$
	be commuting finite order automorphisms of $\fg$ with periods
	$m_1,\ldots,m_\nu$ respectively. Let $\omega_i$ be a primitive $m_i^{th}$-root of
	unity for $1\leq i\leq\nu$. Then 
	$$
	\fg=\bigoplus_{\lam\in\bbbz^n}\fg^{\bar\lam},
	$$ where for $\lam=(\lam_1,\ldots,\lam_\nu)\in\bbbz^\nu$,
	$\bar\lam:=(\bar\lam_1,\ldots,\bar\lam_\nu)$ with 
	$\bar\lam_j:=\lam_j+m_j\bbbz\in\bbbz_{m_j}$, and
	$$
	\fg^{\bar\lam}=\{x\in\fg\mid\sg_j(x)=\omega^{\lam_j}_{j}x\hbox{
		for }1\leq j\leq\nu\}.
	$$
Let $\pi_{\bar{\lam}}:\fg\rightarrow\fg^{\bar{\lam}}$ denote the projection onto $\fg^{\bar{\lam}}$. The subalgebra
	\begin{equation}\label{ref}
		M(\fg,\sg_1,\ldots,\sg_\nu):=\bigoplus_{\lam\in\bbbz^n}\pi_{\bar{\lam}}(\fg)\otimes
		z^\lam=
\bigoplus_{\lam\in\bbbz^n}\fg^{\bar\lam}\otimes
		z^\lam
	\end{equation}
	of $L(\fg,\aa)=\fg\otimes\aa$ is called the {\it $\nu$-step multi-loop
		algebra} of $\sg_1,\ldots,\sg_\nu$ based on $\fg$. 

\end{DEF}

Consider the multi-loop algebra $M(\fg,\sg_1,\ldots,\sg_\nu)$ and
assume that automorphisms $\sg_i$'s preserve $\fh$.
As before, we set $\v:=\Lam\otimes\bbbc$. We extend each automorphism $\sg_i$ of $\fg$ to $\hat{L}(\fg,\aa)=\fg\otimes\aa\oplus\v\oplus\v^\star$ by
\begin{equation}\label{eq13}
	\bar\sg_i(x\otimes z^\lam+d+\gamma)=\sg_i(x)\otimes\omega_i^{-\bar\lam_i}z^\lam+d+\gamma,
\end{equation}
for 
$x\in\fg$, $\lam=(\lam_1,\ldots,\lam_\nu)\in\Lam$, $d\in\v$ and
$\gamma\in\v^\star$. Then
$$\hat{M}:=\hat{M}(\fg,\sg_1,\ldots,\sg_\nu):=M(\fg,\sg_1,\ldots,\sg_\nu)\oplus\v\oplus \v^\star$$  is the fixed point subalgebra
of $\hat{L}(\fg,\aa)$ under automorphisms $\bar\sg_1, \ldots,\bar\sg_\nu$. We set $\hat\fh:=(\fh^{\bar 0}\otimes 1)\oplus\v\oplus\v^\star$.

{From Section \ref{Chevalley automorphism and extended affinization} we know that  $\hat{L}(\fg,\aa)$ has a root space decomposition with respect to $\fh\otimes 1\oplus\v\oplus\v^\star$. So it has a weight space decomposition with respect to $\hat{\fh}\subseteq\fh\otimes 1\oplus\v\oplus\v^\star$. Then it follows from \cite[Proposition II.1]{MP95} that $\hat{M}$ also has a weight space decomposition with respect to $\hat{\fh}$. We denote the set of weights of $\hat{M}$ with respect to $\hat{\fh}$ by $\hat{R}$.} Recall from Section \ref{sub-sec-EALA} that the form on $\fh$ transfers to a form on $\fh^\star$, and also that  the form on $\fg$ extends to a non-degenerate symmetric invariant form on $\hat{L}(\fg,\aa)$, see
(\ref{formloop}).

We denote the set of fixed points of
$\sg_1,\ldots,\sg_i$ on $\fg$ and $\fh$ by
$\fg^{\sg_1,\ldots,\sg_i}$ and $\fh^{\sg_1,\ldots,\sg_i}$, respectively. 
Finally,  we denote the map $\pi=\pi_0$ defined by (\ref{eqram}) corresponding to $\sg_i$ by $\pi^i$.

\begin{thm}\label{thm12}
	Suppose that $(\fg,\fm,\fh)$ is an extended affine Lie algebra with root system $R$
	and $(\sg_1,\ldots,\sg_\nu)$ is an ordered tuple of commuting finite order automorphisms of $\fg$ satisfying for
	$1\leq i\leq\nu$,
	
	- $\sg_i(\fh)=\fh,$
	
	- $(\sg_i(x), \sg_i(y))=(x, y)$ for $x, y\in\fg,$
	
	- {$ C_{\fg^{\sg_1,\ldots,\sg_i}}(\fh^{\sg_1,\ldots,\sg_i})\sub\fh^{\sg_1,\ldots,\sg_i},$}
	
	- {$(\a,\pi^1\cdots\pi^\nu(\a))\neq 0$, for some $\a\in R.$}
	
	\noindent {Then $(\hat{M},\fm,\hat\fh)$ is an extended affine Lie algebra with root system $\hat R$, where  $\fm$ is the form (\ref{formloop}) on
		$\hat{L}(\fg,\aa)$ restricted to $\hat{M}$. Moreover
$\hat R=\pi^1\cdots\pi^\nu(R)$, mod Rad$\fm$.}
\end{thm}

\proof
We recall that the form on $\hat{L}(\fg,\aa)$  defined by (\ref{formloop}) is symmetric, invariant and non-degenerate. 
We use induction
on the number of automorphisms $\sg_i$. For this we fix a $\bbbz$-basis $\d_1,\ldots,\d_\nu$ of $\Lam$. Then for $1\leq k\leq\nu$, we set
$$\begin{array}{l}
	\Lam_k:=\sum_{i=1}^k\bbbz\d_i,\\
	\v_k:=\Lam_k\otimes\bbbc,\\
	M_k:=M(\fg,\sg_1,\ldots,\sg_k),\\
	\hat{M}_k:=M_k\oplus\v_k\oplus\v_k^\star,
\end{array}
\andd
\begin{array}{l}
	\fg_k:=\fg^{\sg_1,\ldots,\sg_k},\\
	\fh_k:=\fh^{\sg_1,\ldots,\sg_k},\\
	\hat{\fh}_{k}:=
	\fh_k\otimes 1\oplus\v_k\oplus\v_k^\star.
\end{array}
$$ 
We consider the natural embeddings $\Lam_k\hookrightarrow\Lam_{k+1}$, $\v_k\hookrightarrow\v_{k+1}$. We note that each automorphism
$\bar\sg_i$ of $\hat{M}$ defined by (\ref{eq13}) restricts to an automorphism
of $\hat{M}_k$. {We denote the set of weights of $\hat{M}_{k}$ with respect to $\hat{\fh}_{k}$ by $\hat{R}_{k}$.}

Before starting the induction, we claim that for each $1\leq k\leq\nu$,
 \begin{equation}\label{eqram2}
(\b_k,\pi^k(\b_k))\not=0,\hbox {for }\b_k:=\pi^1\cdots\pi^{k-1}(\a),
\end{equation} 
where $\a$ is as in the statement. {In fact by \ref{eqram3} we have 
$$(\b_k,\pi^k(\b_k))=(\pi^1\cdots\pi^{k}(\a),\pi^1\cdots\pi^k(\a)).$$
On the other hand applying (\ref{eqram3}) repeatedly we see that
\begin{eqnarray*}
	(\pi^1\cdots\pi^\nu(\a),\pi^1\cdots\pi^\nu(\a))&=&
	(\pi^2\cdots\pi^\nu(\a),\pi^1\cdots\pi^\nu(\a))\\
	&=&
	(\pi^2\cdots\pi^\nu(\a),\pi^2\pi^1\pi^3\cdots\pi^\nu(\a))\\
	&=&
	(\pi^3\cdots\pi^\nu(\a),\pi^3\pi^1\pi^2\pi^4\cdots\pi^\nu(\a))\\
	&\vdots&\\
	&=&(\a,\pi^1\cdots\pi^\nu(\a))\neq0.
\end{eqnarray*}
Now we get $(\pi^1\cdots\pi^{k}(\a),\pi^1\cdots\pi^k(\a))\neq0$ since each $\pi^i$ maps Rad$\fm$ into Rad$\fm$. So $(\b_k,\pi^k(\b_k))\neq0$, as was claimed.} 

{Now if we have just one automorphism, then by Theorem \ref{thm09}, {$\hat{M_1}=\hat{L}(\fg,\sg_1)$ is an extended affine Lie algebra with root system $\hat{R}_1=\pi^1(R)$, mod Rad$\fm$}.  To proceed with the induction step, assume next that $(\hat{M}_{k-1},\fm,{\hat\fh}_{k-1})$ is an extended affine Lie algebra with root system
$\pi^1\cdots\pi^{k-1}(R)$, mod Rad$\fm$. We  note that our assumptions on $\sg_i$'s guarantee that,}

- {$\bar\sg_k(\hat{\fh}_{k-1})=\hat{\fh}_{k-1},$}

- {$(\bar\sg_k(x), \bar\sg_k(y))=(x, y)$ for $x, y\in\hat{M}_{k-1},$}

- {$C_{\hat{M}_{k-1}^{\bar\sg_k}}(\hat{\fh}_{k-1}^{\bar\sg_k})\sub\hat{\fh}_{k-1}^{\bar\sg_k}$.}

\noindent
{These together with (\ref{eqram2}) therefore implies that the automorphism ${\bar\sg}_k$ and the extended affine Lie algebra
$\hat{M}_{k-1}$ satisfy (\ref{eq1}). Now since $\hat{M}_{k}=
\hat{L}_{\rho_k}(\hat{M}_{k-1},\bar\sg_k),$
where here  $\rho_k:\bbbz\d_k\rightarrow\bbbz_{m_k}$ is given by
$\lam\d_k\mapsto\bar{\lam}$, we conclude from
Theorem \ref{thm09} that $(\hat{M}_k,\fm,\hat\fh_k)$ is an extended affine Lie algebra with the required root system.}\qed

\begin{exa}\label{exa13}{Let $(\fg,\fm,\fh)$ be the finite dimensional simple Lie algebra of type $E_7$ and $\{h_{i},E_{\pm i}\mid 1\leq i\leq7\}$ be the Chevalley generators of $\fg$. We consider the order $2$ diagonal automorphism of $\fg$ defined on the Chevalley generators by
		$$\sg_{1}(E_{\pm 7})=-E_{\pm 7},\quad\sg_{1}(E_{\pm i})=E_{\pm i}, \quad\text{for}\; i\neq 7,$$
		and also the order $2$ automorphism $\sg_2$ of $\fg$ defined on the Chevalley generators by
		$$\begin{array}{c}
			\sg_{2}(E_{\pm 1})=E_{\pm 6},\quad\sg_{2}(E_{\pm 3})=E_{\pm 5},\\
			\sg_{2}(E_{\pm 7})=E_{\mp\theta},\quad\sg_{2}(E_{\pm i})=E_{\pm i}, \quad\text{for}\; i=2,4,
		\end{array}$$
		where $\theta$ is the highest root and $E_{\theta},E_{-\theta}$ are root vectors corresponding to highest and lowest roots, respectively. Clearly $\sg_i(\fh)=\fh$, for $i=1,2$. Since $(\fg_{\a},\fg_{\b})=\{0\}$ unless $\a+\b=0$, then 
		$$(\sg_i(x), \sg_i(y))=(x, y), \quad\text{for} \;x, y\in\fg,\; i=1,2.$$
		Since $\sg_{1}(\a)=\a$, for all $\a\in\fh^*$, we get from \cite[Proposition 3.25]{ABP02} that $C_{\fg^{\sg_1}}(\fh^{\sg_1})\subseteq\fh^{\sg_1}$. Next we show that $C_{\fg^{\sg_{1},\sg_{2}}}(\fh^{\sg_{1},\sg_{2}})\subseteq\fh^{\sg_{1},\sg_{2}}$. First we note that $\fg^{\sg_{1},\sg_{2}}$ is the simple Lie algebra of type $F_4$,  see \cite[Section 4]{Van01}. So it is enough to prove that $\fh^{\sg_{1},\sg_{2}}=\fh^{\sg_2}$ is a Cartan subalgebra of $\fg^{\sg_{1},\sg_{2}}$. By \cite[Lemma 8.1]{Kac90}, $C_{\fg}(\fh^{\sg_2})$ is a Cartan subalgebra of $\fg$. Since $\fh\subseteq C_{\fg}(\fh^{\sg_2})$ we have $\fh=C_{\fg}(\fh^{\sg_2})$. If $\fh^\prime\supseteq\fh^{\sg_2}$ is a Cartan subalgebra of $\fg^{\sg_{1},\sg_{2}}$, then $\fh^\prime\subseteq C_{\fg}(\fh^{\sg_2})=\fh$. So $\fh^\prime\subseteq\fh\cap\fg^{\sg_{1},\sg_{2}}=\fh^{\sg_{2}}$. Hence $\fh^{\sg_{2}}=\fh^\prime$ and $\fh^{\sg_{2}}$ is a Cartan subalgebra of $\fg^{\sg_{1},\sg_{2}}$. On the other hand $(\a_2,\pi^2\pi^1(\a_2))=(\a_2,\a_2)\neq0$. {Thus the ordered pair $(\sg_1,\sg_2)$ satisfies conditions of Theorem \ref{thm12} and so} $(\hat{M},\fm,\hat\fh)$ is an extended affine Lie algebra. {In fact $\hat{M}$ is a nullity $2$ extended affine Lie algebra of type $F_4$.}}
\end{exa}
	
\begin{rem}
{Section $4$ of \cite{Van01} provides us more examples satisfying the conditions of Theorem \ref{thm12}. In fact by a similar discussion as in Example \ref{exa13} one can apply Theorem \ref{thm12} to simple Lie algebras with pairs of automorphisms given in $4.3$ of \cite{Van01} to produce more EALAs of nullity $2$.}
\end{rem}

{The following definition generalizes the notion of a Chevalley pair for an automorphism Defined in \ref{chev1} to the case of a tuple of commuting finite order automorphisms.} 

\begin{DEF}
{Let $(\fg,\fm,\fh)$ be an extended affine Lie algebra and $\underline{{\mathbf{\sg}}}=(\sg_{1},\ldots,\sg_{\nu})$ be a $\nu$-tuple of commuting finite order automorphisms of $\fg$. A pair $(\tau,\psi)$ of finite order automorphisms of $\fg$ is called a {\it $\underline{{\mathbf{\sg}}}$-Chevalley pair} if}

	- $\tau$ is a Chevalley automorphism, 
	
	- $\sg_i\tau=\tau\sg_i$, for $i=1,\ldots,\nu$
	
	- $\psi(\fg^{\bar\lam})=\fg^{-\bar\lam}$, for $\lam\in\Lam$,
	
	- $\psi$ preserves the form,
	
	- $\psi(h)=h$ for $h\in\fh^{\bar{0}}$.
\end{DEF}

\begin{pro}\label{lem12}
	{Assume that $(\fg,\fm,\fh)$ is an extended affine Lie algebra and the multi-loop algebra $\hat{M}(\fg,\underline{\sg})$ is an extended affine Lie algebra obtained from Theorem \ref{thm12}. Assume that $\fg$ admits a $\underline{{\mathbf{\sg}}}$-Chevalley pair $(\tau,\psi)$. Then  the assignment
		$$\pi_{\bar\lam}(x)\otimes a^\lam+\gamma+d\mapsto \psi\tau(\pi_{\bar\lam}(x))\otimes a^{-\lam}-\gamma-d,$$ 
		($x\in\fg,$ $\lam\in\Lam,$ $\gamma\in\v,$ $d\in\v^\star$) defines a finite order automorphism ${\bar\tau}_{\psi}$ of $\hat\ll(\fg,\aa)$ which restricts to a Chevalley automorphism of $\hat{M}(\fg,\underline{{\mathbf{\sg}}})$.}
\end{pro}

\proof
{The proof is just as that of Proposition \ref{pro115}.}

\begin{exa}\label{newex1}
	{Let $\fg$ and $(\sg_1,\sg_2)$ be as in Example \ref{exa13}.
		Let $\tau$ be the Chevalley automorphism of $\fg$ such that
		$$\tau(E_{\pm i})=-E_{\mp i},\quad\tau(h_{i})=-h_{i}, \quad\text{for}\; i=1,\cdots,7.$$
		Clearly $\tau\sg_{i}=\sg_{i}\tau$, for $i=1,2$ and $(\tau,id)$ is a $\underline{{\mathbf{\sg}}}$-Chevalley pair with $\underline{{\mathbf{\sg}}}=(\sg_{1},\sg_{2})$. So by Proposition \ref{lem12}, $\bar{\tau}_{id}$ is a Chevalley automorphism for $\hat{M}$.
	}
\end{exa}

\begin{rem} 
	One knows that the centerless cores of almost all  extended affine Lie algebras can be realized as multi-loop algebras based on a finite dimensional central simple Lie algebra; the exception is for type
	$A_\ell$ (see \cite{ABFP08} and \cite{Ne04-1}). The results in Sections \ref{Chevalley automorphism and extended affinization} and \ref{chevalley basis and multi-loop realization} show that for an extended affine Lie algebra obtained from a multi-loop realization (as in Theorem \ref{thm12}), a Chevalley automorphism can be achieved starting from a Chevalley automorphism for  a finite dimensional central simple Lie algebra. Then this together with {the results of Section \ref{Integral structure of the core}} can be used to construct a {$\bbbz$-form} for the core of the corresponding extended affine Lie algebra.  
\end{rem}

\section{Applications}\label{sec:appl}
\setcounter{equation}{0}

In this section we present some natural applications of the $\bbbz$-forms obtained for EALAs in the previous sections.\ We always consider those EALAs that admit a $\bbbz$-form in this section.\ First we use these $\bbbz$-forms to define extended affine Lie algebras over fields of positive characteristic then we define certain groups associated to them.

\subsection{EALAs over arbitrary fields}\label{subsecAlarbF}
\setcounter{equation}{0}
Extended affine Lie algebras are originally defined over the field of complex numbers while, up to a tameness condition, the theory can also be developed over arbitrary fields of characteristic zero (see \cite{Ne04}). Here we define EALAs over arbitrary fields via $\bbbz$-forms that is a generalization of the same procedure in the classical context.

For this, we show how one can use the {integral structure (Chevalley construction)} constructed in Section {\ref{Integral structure of the core}} to reduce the coefficients of an extended affine Lie algebra over $\mathbb{C}$ to arbitrary fields.\ 

Assume that $(\fg,\fm,\fh)$ is a tame extended affine Lie algebra of reduced type and of rank $>1$.
We recall from Section \ref{Integral structure of the core} that
the core $\fg_c$ of $\fg$ is equipped with an integral structure $\fg^\bbbz_c$ which is unique up to isomorphism, see Theorem \ref{simply}.

\begin{DEF}\label{defK-form}
	Let $\mathbb{F}_{p}$ be the prime field of characteristic $p$.\ Since $\fg_{c}^{\bbbz}$  is closed under the Lie bracket, $\fg_{c}^{\mathbb{F}_{p}}:=\fg_{c}^{\bbbz}\otimes_{\bbbz}\mathbb{F}_{p}$ has a Lie algebra structure in a natural way.\ Now for any field extension $\mathbb{K}$ of $\mathbb{F}_{p}$, we define $\fg_{c}^{\mathbb{K}}:=\fg_{c}^{\mathbb{F}_{p}}\otimes_{\mathbb{F}_{p}}\mathbb{K}\cong\fg_{c}^{\bbbz}\otimes_{\bbbz}\mathbb{K}$.\ Considering Remark \ref{remned}, we call (the natural extension of) $\fg_{c}^{\mathbb{K}}$ an {\it extended affine Lie algebra over $\mathbb{K}.$} 
\end{DEF}

\subsection{Groups of extended affine Lie type over fields (Adjoint type)}\label{GEALT ad-form}
\setcounter{equation}{0}	
In \cite{AP19} we studied certain groups associated to tame reduced extended affine Lie algebras.\ A method used in \cite{AP19} was the integration procedure (see e.g., \cite{Kac2,Kac85,Ma18}).\ This integration is not applicable while working with fields of positive characteristic.\ To overcome this issue, one can use the $\mathbb{Z}$-form $\fg_c^\bbbz$ (see Theorem~\ref{thmnew5}) to construct groups of extended affine Lie type over arbitrary fields.\ Also recall from \cite[Proposition 4.22]{AP19} that the groups that are associated to $\fg$ and $\fg_c$ via integration are isomorphic.\ Therefore, in this section, we only consider $\fg_c$ and the group associated to it.\ 

Throughout this subsection, we assume $\g$ to be an EALA with the conditions in \cite{AP19} such that there exists a group associated to it via integration.\
Let $G$ denote the group associated to $\fg_c$.\ More precisely, if $\hat{G}:=\ast_{\alpha\in R^{\times}}\g_{\alpha}$ denotes the free product of abelian additive groups $\g_{\alpha}$, then $G$ is the quotient of $\hat{G}$ over the intersection of  kernels of all integrable representations of $\g$ (see e.g., \cite{Kry95}, \cite{Kac2} for more details).\ Recall from \cite[\S4.2]{AP19} that the adjoint form of $G$ is defined as follows.
\begin{equation}\label{Gad}
	\text{Ad}(G):=\langle \text{exp}(\text{ad}(cx_{\a}))~|~ \a\in R^{\times}, c\in\bbbc\rangle\leq G\leq\text{Aut}(\g).
\end{equation}
Also, $\fg_c$ is integrable in the sense of \cite{Kac2} and we have 
\begin{equation}\label{eq:weylgraction}
	\text{ad}(\text{exp(\text{ad}(y))}(x))=\text{exp}(\text{ad}(x))\text{ad}(y)\text{exp}(-\text{ad}(x)).
\end{equation}

In Subsection~\ref{subsecAlarbF}  we saw that $\fg_{c}^{\mathbb{K}}$ only depends on $(\fg, \mathbb{K})$.\ Therefore, here we consider a fixed $\bbbz$-form and start constructing the desired groups associated to $\fg_{c}^{\mathbb{K}}$ by turning every $\fg_{c}$-module into a $\fg_{c}^{\mathbb{K}}$-module.\ For this, we start off by considering  $\fg_{c}$ as a  $\fg_{c}$-module via the adjoint action.\ Let $\mathfrak{sl}_2^{\a}$ denote the rank one subalgebra of $\fg_c$ for every non-isotropic root $\a$.\ Then recall that by \cite[Proposition 4.13]{AP19} the adjoint map from $\mathfrak{sl}_2^{\a}$ to $\fg_c$ integrates to the following map:
\begin{equation}\label{eq:ProSL2modii00}
	\\ \text{Ad}^{\a}:\SL(\bbbc)^{\a}\to \text{Ad}(G)\leq\text{Aut}(\mathfrak{g}_c),
\end{equation}	
such that for all $t\in \bbbc$
\begin{equation}\label{eq:ProSL2modii01}
	\text{Ad}^{\a}\left(\left(
	\begin{array}{cc}
		1 & t \\
		0 & 1 \\
	\end{array}
	\right)\right)=\exp~t~\text{ad}(x_{\a}),~~~\text{Ad}^{\a}\left(\left(
	\begin{array}{cc}
		1 & 0 \\
		t & 1 \\
	\end{array}
	\right)\right)=\exp~t~\text{ad}(x_{-\a}).
\end{equation}
But from the construction of $\fg_c^\bbbz$, it is clear that if we restrict the above argument to $\mathfrak{sl}_2^{\a}\cap\fg_c^\bbbz$, we have 
\begin{equation}\label{eq:ProSL2modii00Z}
	\\ \text{Ad}_{\bbbz}^{\a}:\SL(\bbbz)^{\a}\to \text{Ad}(G)_{\bbbz}\leq\text{Aut}(\mathfrak{g}_c),
\end{equation}	
where 
\begin{equation}\label{GZ06}
	\text{Ad}(G)_{\bbbz}:=\langle \text{exp}(\text{ad}(cx_{\a}))~|~ \a\in R^{\times}, c\in\bbbz\rangle<\text{Ad}(G)\leq\text{Aut}(\fg_c).
\end{equation}
Note that the image of $\text{Ad}_{\bbbz}^{\a}$ lies in $\text{Ad}(G)_{\bbbz}$ since by Lemma~\ref{lemassump2}, $\fg_c^\bbbz$ is stable under $\text{exp}(\text{ad}(x_{\a}))$ {for} $\a\in R^\times$.\ 

In order to define the adjoint form on arbitrary fields, we consider $\fg_{c}^{\mathbb{K}}$ as a $\fg_{c}^{\mathbb{K}}$-module via the adjoint action as described above where $\bbbk$ is considered as an extension of $\bbbf_{p}$ for some prime number $p$.\ More precisely, the natural adjoint action $\text{ad}^{\bbbk}$ of $x\otimes r\in\fg_{c}^{\mathbb{K}}$ on $\fg_{c}^{\mathbb{K}}$ is defined by $$\text{ad}(x)(v)\otimes r\cdot s$$
for every $v\otimes s \in\fg_{c}^{\mathbb{K}}\cong\fg_{c}^{\bbbz}\otimes_{\bbbz}\mathbb{K}.$\ Now by the same integration procedure as above, we obtain the adjoint form over $\mathbb{K},$ namely, we have 
\begin{equation}\label{eq:ProSL2modii00K}
	\\ \text{Ad}_{\bbbk}^{\a}:\SL(\bbbk)^{\a}\to \text{Ad}(G)_{\bbbk}\leq\text{Aut}(\mathfrak{g}_{c}^{\mathbb{K}}),
\end{equation}	
where 
\begin{equation}\label{GKK}
	\text{Ad}(G)_{\bbbk}:=\langle \text{exp}(\text{ad}^{\bbbk}(x_{\a}\otimes c))~|~ \a\in R^{\times}, c\in\bbbk\rangle\leq\text{Aut}(\fg_{c}^{\mathbb{K}}).
\end{equation}
We nominate the resulting group $\text{Ad}(G)_{\bbbk}$ as the adjoint form of the group associated to $\fg_c$ over $\bbbk.$\ Note again that the image of $\text{Ad}_{\bbbk}^{\a}$ lies in $\text{Ad}(G)_{\bbbk}$ by Lemma~\ref{lemassump2}.

\begin{rem}
	The above procedure suggests a similar construction of groups of extended affine Lie type associated to integrable highest weight representations (for such representations see e.g., \cite{GZ06,Z11}). We intent to investigate this in a separate study.\ Furthermore, note that since the representation theory of extended affine Lie algebras has not been comprehensively characterized by integrable highest weight representations, finding a $\bbbz$-form in any integrable representation is another open problem and hence, at this stage, we are not able to define a group of the universal type $G$ over arbitrary fields.\  
\end{rem}

\begin{rem}
	Note that the above construction works for rings as well but in the case that the underlying ring is not Euclidean, instead of $\SL$ one has to work with the group generated by transvections, namely $E_{2}$.  
\end{rem}

Next we investigate groups of extended affine Lie type over finite fields.\ For this first recall from \cite[Definition 4.1]{AP19} that for a unital associative ring $A$, the Steinberg group of type $\rd$ over $A$ is denoted by $\St_{\dot{R}}(A)$ where $\rd$ is as in Section~\ref{sub-sec-EALA}.\ In particular, if $\bbbc_{\sigma}$ denotes the coordinate ring of $\g$ as defined in Section~\ref{Chevalley automorphism and extended affinization} for some 2-cocycle $\sigma$, then the Steinberg group of type $\rd$ over $\bbbc_{\sigma}$ is denoted by $\St_{\dot{R}}(\bbbc_{\sigma})$.\  We further assume that $\sigma$ is elementary, meaning, the image of $\sigma$ lies in $\{\pm1\}$.\ Therefore, one can readily define $\bbbk_{\sigma}$ for any field $\bbbk.$ 

\begin{lem}\label{lemStfin}
	The Steinberg group $\St_{\dot{R}}(\bbbk_{\sigma})$ for an elementary 2-cocycle $\sigma$ over a finite field $\bbbk$ {with characteristic at least 5,} is finitely generated.
\end{lem}

\proof
	Let $\nu$ be the nullity of $\g.$\ {One knows from \cite[Chapter II]{AABGP97} that the lattice generated by isotropic roots contains a basis $\delta_{i}\in R^0$ ($1\leq i\leq\nu$)}, {for simplicity of the notation in calculations, let $a_{i}$ denote a multiplicative change of variables corresponding to $\delta_{i}$ in the coordinate ring $\bbbk_{\sigma}$, e.g., we have $a_{i}^{n}a_{j}^{m}=n\delta_{i}+m\d_j$ in this setting.}\ Moreover, since the 2-cocycle is assumed to be elementary, the following calculations are accurate up to a sign and without loss of generality we assume that $\sigma\cong1$.
		
 First assume $\text{rank}(\dot{R})=1.$\ We claim that the finite set 
 \[
 \mathcal{{X}}:=\{\hat{x}_{\dot{\alpha}}(\bbbk \prod_{i=1}^{\nu} a^{n_{i}}_{i})~|{~n_{i}\in\{\pm1,0\},~\dot{\a}\in\dot{R}}\},
 \]
 generates $\St_{\dot{R}}(\bbbk_{\sigma})$.\ First we show that the higher powers $\hat{x}_{\dot{\alpha}}(a^{\pm n}_{i})$, ($n>1,1\leq i\leq\nu $) are generated by $\mathcal{{X}}$.\ Higher powers of negative sign can be generated similarly.
 
 Recall from \cite{AP19} that for every unit $a\in\bbbk_{\sigma}$  
 \begin{equation*}
 \hat{n}_{\dot{\alpha}}(a):=\hat{x}_{\dot{\alpha}}(a)\hat{x}_{-\dot{\alpha}}(-a^{-1})\hat{x}_{\dot{\alpha}}(a).
 \end{equation*}
 Now by \textbf{(StF$\text{2}'$)} in \cite[Definition 4.1]{AP19} we have
	\[
	\hat{n}_{\dot{\alpha}}(a^{-1}_{i})\hat{x}_{\dot{\alpha}}(-1)\hat{n}_{\dot{\alpha}}(a^{-1}_{i})^{-1}=\hat{x}_{-\dot{\alpha}}(a_{i}^{2}),
	\]
	and
	\[
	\hat{n}_{\dot{\alpha}}(a^{-1}_{i})\hat{x}_{\dot{\alpha}}(-a_{i})\hat{n}_{\dot{\alpha}}(a^{-1}_{i})^{-1}=\hat{x}_{-\dot{\alpha}}(a_{i}^{3}).
	\]
	Then 
	\[
	\hat{n}_{-\dot{\alpha}}(1)\hat{x}_{-\dot{\alpha}}(a_{i}^{2})\hat{n}_{-\dot{\alpha}}(1)^{-1}=\hat{x}_{\dot{\alpha}}(-a_{i}^{2}),
	\]
	and
	\[
	\hat{n}_{-\dot{\alpha}}(1)\hat{x}_{-\dot{\alpha}}(a_{i}^{3})\hat{n}_{-\dot{\alpha}}(1)^{-1}=\hat{x}_{\dot{\alpha}}(-a_{i}^{3}).
	\]
	Higher powers can be generated inductively with the above algorithm.
For generating elements $\hat{x}_{\dot{\alpha}}(a_{i}^{n}a_{j}^{m})$ ($n,m\in\mathbb{Z}, \dot{\alpha}\in\rd$) by \textbf{(StF$\text{2}'$)} in \cite[Definition 4.1]{AP19} we have
	\[
\hat{n}_{\dot{\alpha}}(a^{-n}_{i})\hat{x}_{\dot{\alpha}}(-a_{j})\hat{n}_{\dot{\alpha}}(a^{-n}_{i})^{-1}=\hat{x}_{-\dot{\alpha}}(a_{i}^{2n}a_{j}),
\]	
and
 	\[
 \hat{n}_{\dot{\alpha}}(a^{-n}_{i})\hat{x}_{\dot{\alpha}}(-a_{i}a_{j})\hat{n}_{\dot{\alpha}}(a^{-n}_{i})^{-1}=\hat{x}_{-\dot{\alpha}}(a_{i}^{2n+1}a_{j}).
 \]	
 Therefore, elements $\hat{x}_{-\dot{\alpha}}(a_{i}^{n}a_{j})$ ($n\in\mathbb{Z}$) are generated by $\mathcal{{X}}$ and hence by the relations
 \[
 \hat{n}_{-\dot{\alpha}}(1)\hat{x}_{-\dot{\alpha}}(a_{i}^{n}a_{j})\hat{n}_{-\dot{\alpha}}(1)^{-1}=\hat{x}_{\dot{\alpha}}(-a_{i}^{n}a_{j}),
 \]
 elements $\hat{x}_{\dot{\alpha}}(a_{i}^{n}a_{j})$ ($n\in\mathbb{Z}$) are generated by $\mathcal{{X}}$ as well.\ 
 
 Again by \textbf{(StF$\text{2}'$)} in \cite[Definition 4.1]{AP19} we have
 \[
 \hat{n}_{\dot{\alpha}}(a^{-m}_{j})\hat{x}_{\dot{\alpha}}(-a_{i}^{n})\hat{n}_{\dot{\alpha}}(a^{-m}_{j})^{-1}=\hat{x}_{-\dot{\alpha}}(a_{i}^{n}a_{j}^{2m}),
 \]	
 and
 \[
 \hat{n}_{\dot{\alpha}}(a^{-m}_{j})\hat{x}_{\dot{\alpha}}(-a_{i}^{n}a_{j})\hat{n}_{\dot{\alpha}}(a^{-m}_{j})^{-1}=\hat{x}_{-\dot{\alpha}}(a_{i}^{n}a_{j}^{2m+1}).
 \]	
 Therefore, elements $\hat{x}_{-\dot{\alpha}}(a_{i}^{n}a_{j}^{m})$ ($n,m\in\mathbb{Z}$) are generated by $\mathcal{{X}}$ and hence by the relations
 \[
 \hat{n}_{-\dot{\alpha}}(1)\hat{x}_{-\dot{\alpha}}(a_{i}^{n}a_{j}^{m})\hat{n}_{-\dot{\alpha}}(1)^{-1}=\hat{x}_{\dot{\alpha}}(-a_{i}^{n}a_{j}^{m}),
 \]
 elements $\hat{x}_{\dot{\alpha}}(a_{i}^{n}a_{j}^{m})$ ($n\in\mathbb{Z}$) are generated by $\mathcal{{X}}$ as well.\ Now note that the above procedure can be applied to any $n$-tuple ($2\leq n\leq\nu$) of elements $a_{i}$ ($1\leq i\leq\nu$).\ Therefore, all higher powers $\prod_{i=1}^{\nu} a^{n_{i}}_{i}$ ({$n_{i}\in\mathbb{Z}$}) are generated by $\mathcal{{X}}$.
 
	Assume now that $\text{rank}(\dot{R})>1.$\ We claim that the finite set 
	\[
	\mathcal{{X}}:=\{\hat{x}_{\dot{\alpha}}(\bbbk a^{n}_{i})~|~n\in\{\pm1,0\},~1\leq i\leq\nu,~\dot{\a}\in\dot{R}\},
	\]
	generates $\St_{\dot{R}}(\bbbk_{\sigma})$.\ First we show that the higher powers $\hat{x}_{\dot{\alpha}}(a^{\pm n}_{i})$, ($n>1,1\leq i\leq\nu $) are generated by $\mathcal{{X}}$.\ Higher powers with negative sign can be generated similarly.

	By the locality of our arguments, it suffices to only consider rank 2 cases, namely, types $A_{2}$, $B_{2}$ and $G_{2}.$\ Here we only check type $G_{2}$ since the other two types can be checked similarly and they have simpler relations.\ Also, the structure constants appear here which are modulo the characteristic of $\bbbk$.\ {By assuming that the characteristic is at least 5, all the structure constants survive.}\ Let $\dot{\a},\dot{\b}$ be a base for $G_{2}.$\ Then by \textbf{(StF2)} in \cite[Definition 4.1]{AP19} we have
	\[
	(\hat{x}_{\dot{\alpha}}(a_{i}),\hat{x}_{\dot{\beta}}(c^{-1}_{32}a^{-1}_{i}))=\hat{x}_{\dot{\alpha}+\dot{\beta}}(c_{11}c^{-1}_{32})\hat{x}_{2\dot{\alpha}+\dot{\beta}}(c_{21}c^{-1}_{32}a_{i})\hat{x}_{3\dot{\alpha}+2\dot{\beta}}(a_{i}^{2}),
	\]
	and therefore $\hat{x}_{3\dot{\alpha}+2\dot{\beta}}(a_{i}^{2})$ is generated.\ Similarly $\hat{x}_{3\dot{\alpha}+2\dot{\beta}}(a_{i}^{-2})$ can also be generated.\ Note that the subgroup $\dot{W}:=\langle \hat{n}_{\dot{\gamma}}(\bbbk)~|~\dot{\gamma}\in\dot{R}\rangle$ is clearly generated by $\mathcal{{X}}.$\ Now by \textbf{(StF3)} in \cite[Lemma 4.2]{AP19} we can move $a_{i}^{2}$ to the entry $\hat{x}_{\dot{\beta}}$ and by \textbf{(StF2)} in \cite[Definition 4.1]{AP19} we have 
	\[
	(\hat{x}_{3\dot{\alpha}+\dot{\b}}(a_{i}),\hat{x}_{\dot{\beta}}(c_{11}^{-1}a^{2}_{i}))=\hat{x}_{3\dot{\alpha}+2\dot{\beta}}(a_{i}^{3}).
	\]
	With the above procedure and by the help of the action of $\dot{W}$ as in \textbf{(StF3)} in \cite[Lemma 4.2]{AP19} all powers of $a^{n}_{i}$ for all the entries corresponding to the long roots, namely {$\pm\dot{\beta}$, $\pm(3\dot{\alpha}+\dot{\beta})$ and $\pm(3\dot{\alpha}+2\dot{\beta})$}, are generated.
	
	For the entries corresponding to the short roots, again by  \textbf{(StF2)} in \cite[Definition 4.1]{AP19} we have
	\[
	(\hat{x}_{\dot{\alpha}}(a_{i}),\hat{x}_{\dot{\beta}}(c_{21}^{-1}))=\hat{x}_{\dot{\alpha}+\dot{\beta}}(c_{11}c_{21}^{-1}a_{i})\hat{x}_{2\dot{\alpha}+\dot{\beta}}(a_{i}^{2})\hat{x}_{3\dot{\alpha}+2\dot{\beta}}(c_{32}c_{21}^{-1}a_{i}^{3}).
	\]
	We showed above that $\hat{x}_{3\dot{\alpha}+2\dot{\beta}}(c_{32}c_{21}^{-1}a_{i}^{3})$ is generated by $\mathcal{X}${, hence} $\hat{x}_{2\dot{\alpha}+\dot{\beta}}(a_{i}^{2})$ is also generated by $\mathcal{X}$ and by the action of $\dot{W}$ as in \textbf{(StF3)} in \cite[Lemma 4.2]{AP19}, the power 2, namely $a_{i}^{2}$ is generated for all the entries corresponding to the short roots; {$\pm\dot{\a}$, $\pm(\dot{\alpha}+\dot{\beta})$ and $\pm(2\dot{\alpha}+\dot{\beta})$} and, in particular, for the entry $\hat{x}_{\dot{\alpha}}$.\ Again by \textbf{(StF2)} in \cite[Definition 4.1]{AP19} we have
	\[
	(\hat{x}_{\dot{\alpha}}(a^{2}_{i}),\hat{x}_{\dot{\a}+\dot{\beta}}(c_{11}^{-1}a_{i}))=\hat{x}_{2\dot{\alpha}+\dot{\beta}}(a^{3}_{i})\hat{x}_{3\dot{\alpha}+\dot{\beta}}(c_{21}c_{11}^{-1}a_{i}^{5}).
	\]
	{Since} we showed that $\hat{x}_{3\dot{\alpha}+\dot{\beta}}(c_{21}c_{11}^{-1}a_{i}^{5})$ is generated by $\mathcal{X}$, therefore, it follows from the above equation that $\hat{x}_{2\dot{\alpha}+\dot{\beta}}(a^{3}_{i})$ is also generated by $\mathcal{X}.$ Now by repeating this procedure and by the help of the action of $\dot{W}$ as in \textbf{(StF3)} in \cite[Lemma 4.2]{AP19} all powers of $a^{n}_{i}$ for all the entries corresponding to the short roots, namely $\dot{\a}$, $\dot{\alpha}+\dot{\beta}$ and $2\dot{\alpha}+\dot{\beta}$, are also generated.\ 
	
	For generating elements $\hat{x}_{\dot{\alpha}}(a_{i}^{n}a_{j}^{m})$ ($n,m\in\mathbb{Z}, \dot{\alpha}\in\rd$) by  \textbf{(StF2)} in \cite[Definition 4.1]{AP19} we have 
	\[
	(\hat{x}_{3\dot{\alpha}+\dot{\b}}(a_{i}^{n}),\hat{x}_{\dot{\beta}}(c_{11}^{-1}a_{i}^{m}))=\hat{x}_{3\dot{\alpha}+2\dot{\beta}}(a_{i}^{n}a_{j}^{m}).
	\]
	Hence by \textbf{(StF3)} in \cite[Lemma 4.2]{AP19} all products $a^{n}_{i}a_{j}^{m}$ for all the entries corresponding to the long roots, namely $\dot{\beta}$, $3\dot{\alpha}+\dot{\beta}$ and $3\dot{\alpha}+2\dot{\beta}$, are generated by $\mathcal{{X}}$.\ Again by \textbf{(StF2)} in \cite[Definition 4.1]{AP19} we have 
	\[
	(\hat{x}_{\dot{\alpha}}(a^{n}_{i}),\hat{x}_{\dot{\a}+\dot{\beta}}(c_{11}^{-1}a_{j}^{m}))=\hat{x}_{2\dot{\alpha}+\dot{\beta}}(a^{n}_{i}a_{j}^{m})\hat{x}_{3\dot{\alpha}+\dot{\beta}}(c_{21}c_{11}^{-1}a_{i}^{2n}a_{j}^{m}).
	\]
	Now since we showed that $\hat{x}_{3\dot{\alpha}+\dot{\beta}}(c_{21}c_{11}^{-1}a_{i}^{2n}a_{j}^{m})$ are generated by $\mathcal{X}$, therefore, it follows from the above equation that $\hat{x}_{2\dot{\alpha}+\dot{\beta}}(a^{n}_{i}a_{j}^{m})$ are also generated by $\mathcal{X}$ and hence in view of \textbf{(StF3)} in \cite[Lemma 4.2]{AP19} all products $a^{n}_{i}a_{j}^{m}$ for all the entries corresponding to the short roots are generated by $\mathcal{X}$ as well.\ The above procedure is valid for every pair $a_{i},a_{j}$ ($1\leq i,j\leq\nu$) and we can extend it to $\nu$-tuples step by step.
	\qed
	
	\begin{rem}
		Note that  in the above lemma we use \cite[Lemma 4.2]{AP19} which assumes that the characteristic is 0  but this condition is redundant.\ 
	\end{rem}

\begin{thm}\label{thmgealtf}
	Let $\text{Ad}(G)_{\bbbk}$ be the adjoint form of the group of extended affine Lie type associated to an EALA $\fg$ of nullity at least one whose coordinate ring 2-cocycle is elementary.\ Then the following hold:
	\begin{itemize}
		\item [(i)] The group $\text{Ad}(G)_{\bbbk}$ is infinite.
		\item[(ii)] If $\bbbk$ is finite {with characteristic at least 5,} then $\text{Ad}(G)_{\bbbk}$ is finitely generated.
	\end{itemize}
\end{thm} 
\begin{proof}
	(i) Since every EALA $\fg$ of nullity at least one contains certain copies of affine Kac-Moody algebras, the group $\text{Ad}(G)_{\bbbk}$ contains copies of the corresponding affine Kac-Moody groups.\ The result then follows from the loop realization of affine Kac-Moody groups.
	
	(ii) Let $\St_{\dot{R}}(\bbbk_{\sigma})$ denote the Steinberg group of type $\rd$ over $\bbbk_{\sigma}$ as in \cite[Definition 4.1]{AP19}.\ By Lemma \ref{lemStfin},  $\St_{\dot{R}}(\bbbk_{\sigma})$ is finitely generated.\ Now by the similar methods used in \cite[Proposition 4.12]{AP19} to define an epimorphism from $\St_{\dot{R}}(\bbbc_{\sigma})$ to $G$, one can show that the assignment 
	\begin{equation}\label{eq:ProisoimageSt01}
		\varpi(\hat{x}_{\dot{\alpha}}(kc_{\delta}))=\text{Ad}(\exp(\text{ad}(c_{\delta}X_{\dot{\alpha}})\otimes k)),
	\end{equation}
	defines an epimorphism from $\St_{\dot{R}}(\bbbk_{\sigma})$ to $\text{Ad}(G)_{\bbbk}$, hence the result. 
	\qed
\end{proof}

This sets a pretext to generalizing the concept of arithmetic groups in non-Archimedean semisimple groups (see \cite{zbMATH05707462}).


\end{document}